\documentclass[12pt]{amsart}
\usepackage{amsmath, amsthm, amssymb}
\usepackage{fullpage}
\usepackage{color}
\usepackage{hyperref}
\usepackage{soul}
\usepackage{enumitem}

\newtheorem{theorem}{Theorem}
\newtheorem{lemma}{Lemma}
\newtheorem{proposition}[lemma]{Proposition}
\newtheorem{corollary}[lemma]{Corollary}
\newtheorem{definition}[lemma]{Definition}
\newtheorem{remark}[lemma]{Remark}
\newtheorem*{conjecture}{Conjecture}
\numberwithin{lemma}{section}

\numberwithin{equation}{section}

\newcommand{\R}{{\mathbb R}}

\newcommand{\Z}{{\mathbb Z}}

\renewcommand{\R}{\mathbb R}
\newcommand{\bL}{\mathbf L}
\newcommand{\bM}{\mathbf M}
\newcommand{\bP}{\mathbf P}
\newcommand{\bE}{\mathbf E}
\newcommand{\bB}{\mathbf B}
\newcommand{\bR}{\mathbf R}

\newcommand{\du}{\mathfrak{u}}
\newcommand{\dv}{\mathfrak{v}}

\newcommand{\bI}{\mathbf I}
\newcommand{\bJ}{\mathbf J}
\newcommand{\bK}{\mathbf K}

\newcommand{\bu}{{\bar u}}
\newcommand{\bv}{{\bar v}}

\newcommand{\la}{\langle}
\newcommand{\ra}{\rangle}

\newcommand{\ol}{\overline}
\newcommand{\ms}{M^\sharp}
\newcommand{\ps}{P^\sharp}
\newcommand{\calR}{\mathcal{R}}

\newcommand{\xih}{{\delta \xi^{\text{hi}}}}
\newcommand{\xim}{{\delta \xi^{\text{med}}}}

\newcommand{\tDelta}{{\tilde \Delta}}

\begin{document}

\title{Global solutions for 1D cubic defocusing dispersive equations: Part I}

\author{Mihaela Ifrim}
\address{Department of Mathematics, University of Wisconsin, Madison}
\email{ifrim@wisc.edu}

\author{ Daniel Tataru}
\address{Department of Mathematics, University of California at Berkeley}
\email{tataru@math.berkeley.edu}

\begin{abstract}

  This article is devoted to a general class
  of one dimensional NLS  problems with a cubic nonlinearity.  The question of obtaining scattering, global in time solutions for such problems has attracted a lot of attention in recent years, and many global well-posedness results have been proved for a number of models under the 
assumption that the initial data is both \emph{small} and \emph{localized}.  However,
except for the completely integrable case,
no such results have been known for small but not necessarily localized initial data. 
  
In this article we introduce a new, nonperturbative method, 
to prove global well-posedness and scattering
for $L^2$ initial data which is \emph{small} and \emph{non-localized}. Our main structural 
assumption is that our nonlinearity is \emph{defocusing}. However, we do not assume that our problem has any exact conservation laws. Our method is based on a robust reinterpretation of the idea of interaction Morawetz estimates, developed almost 20 years ago by the I-team.

In terms of scattering, we prove that our global solutions satisfy both global $L^6$ Strichartz estimates and bilinear $L^2$ bounds. This is a Galilean invariant result, which is new even for the classical   defocusing cubic NLS\footnote{There the global well-posedness was of course
known, but not the Strichartz and bilinear $L^2$ bounds.}. There, by scaling our result also admits a large data counterpart.
\end{abstract}


\maketitle

\setcounter{tocdepth}{1}
\tableofcontents

\section{Introduction}
The question of obtaining scattering, global in time solutions for one dimensional dispersive flows with quadratic/cubic nonlinearities has attracted a lot of attention in recent years, and many global well-posedness results have been proved for a number of models under the assumption that the initial data is both \emph{small} and \emph{localized};  without being exhaustive, see for instance
see \cite{HN,HN1,LS,KP,IT-NLS}.
The nonlinearities in these models are primarily cubic, though the analysis 
has also been extended via normal form methods to problems which also have
nonresonant quadratic interactions; several such examples are  \cite{AD,IT-g,D,IT-c,LLS}, see also further references therein. 

In this article we consider instead the much more difficult case where the initial data
is just \emph{small}, but without any localization assumption. Here it is natural to restrict the analysis
to defocusing problems, as focusing one-dimensional cubic NLS type problems typically admit small solitons and thus, generically, the solutions do not scatter at infinity.  Then one may formulate the following broad 
conjecture:

\begin{conjecture}
One dimensional dispersive flows with cubic defocusing nonlinearities and small initial data have global in time, scattering solutions.
\end{conjecture}


The goal of this article is to prove the first global in time well-posedness result of this type.
 As part of our results, we also prove that our global solutions are scattering at infinity in a very precise, quantitative way, in the sense that they  satisfy both $L^6$ Strichartz estimates and bilinear $L^2$ bounds.
This is despite the fact that the nonlinearity is non-perturbative on large time scales.

\subsection{ Cubic NLS problems in one space dimension}
One of the fundamental one-dimensional dispersive flows in one space 
dimension is the cubic NLS flow, 
\begin{equation}\label{nls3}
i u_t + u_{xx} = \pm u |u|^2,  \qquad u(0) = \du_0.
\end{equation}
Depending on the choice of signs, this comes in a defocusing (+) and a focusing (-) 
flavor. Both of these equations are important not only by themselves, but also 
as model problems for more complex one-dimensional dispersive flows, 
both semilinear and quasilinear. 

The above cubic NLS flow is globally well-posed in $L^2$ both in the focusing and in the defocusing case, though the global behavior differs in the two cases.

Precisely, the focusing problem admits small solitons, so the solutions 
cannot in general scatter at infinity. If in addition the initial data is localized,
then one expects the solution to resolve into a superposition of (finitely many) 
solitons, and a dispersive part; this is called \emph{the soliton resolution conjecture},
and is known to hold in a restrictive setting, via the method of inverse scattering, see e.g. ~\cite{IST-focusing}.

In the defocusing case, the inverse scattering approach also allows one to treat
the case of localized data, and show that global solutions scatter at infinity, see for instance \cite{IST}.
This can also be proved in a more robust way, without using inverse scattering,
under the assumption that the initial data is small and localized, see \cite{IT-NLS} and references therein.
Much less is known in terms of scattering for nonlocalized $L^2$ data. However, if more regularity is assumed for the data, then we have the following estimate due to Planchon-Vega~\cite{PV}, see also the work of Colliander-Grillakis-Tzirakis~\cite{MR2527809}:
\begin{equation}\label{PV-est}
\|u\|_{L^6}^6 + \| \partial_x |u|^2\|_{L^2}^2 \lesssim \| \du_0\|_{L^2}^3 \|\du_0\|_{H^1}.
\end{equation}
This allows one to estimate the $L^6$ Strichartz norm of the solution, i.e. to prove some type of scattering or dispersive decay.

\bigskip

Because of the above considerations, our interest in this paper is in defocusing 
cubic problems. Precisely, we will consider a cubic nonlinear Schr\"odinger equation (NLS)  type model 
in one space dimension
\begin{equation}\label{nls}
i u_t + u_{xx} = C(u,\bar u, u),  \qquad u(0) = \du_0, 
\end{equation}
where $u$ is a complex valued function, $u:\mathbb{R}\times
\mathbb{R}\rightarrow \mathbb{C}$. Here $C$ is a trilinear translation
invariant form, whose symbol $c(\xi_1,\xi_2,\xi_3)$ can always be
assumed to be symmetric in $\xi_1,\xi_3$; see section~\ref{s:multi}
for an expanded discussion of multilinear forms.
The arguments $u,\bar u$ and $u$ of $C$ are chosen so that
our equation \eqref{nls} has the phase rotation symmetry,
$u \to u e^{i\theta}$,
as it is the case in most examples of interest.
The symbol $c(\xi_1,\xi_2,\xi_3)$ will be required to satisfy 
a minimal set of assumptions:

\begin{enumerate}[label=(H\arabic*)]

\item Bounded and regular: 
\begin{equation}\label{c-smooth}
|\partial_\xi^\alpha c(\xi_1,\xi_2,\xi_3)| \leq c_\alpha,
\qquad \xi_1,\xi_2,\xi_3 \in \R,\,   \mbox{  for every  multi-index $\alpha$}.
\end{equation}

\item Conservative: 
\begin{equation}\label{c-conserv}
\Im c(\xi,\xi,\eta) = 0, \qquad \xi,\eta \in \R, \mbox { where } \Im z=\mbox{imaginary part of } z\in \mathbb{C}.
\end{equation}

\item Defocusing: 
\begin{equation}\label{c-defocus}
c(\xi,\xi,\xi) \geq c > 0, \qquad \xi \in \R  \mbox{ and } c \in \mathbb{R^+}.
\end{equation}

\end{enumerate}

In selecting these assumptions we have tried to strike 
a balance between the generality of the result on one hand,
and a streamlined exposition on the other hand. 

The simplest example of such a trilinear form $C$ is of course $C = 1$, which
corresponds to the classical one-dimensional cubic NLS problem.
But this problem is of course completely integrable, and thus has infinitely many conservation laws. In particular global well-posedness
is straightforward, though our $L^6$ Strichartz and bilinear $L^2$ bounds are new even for this problem in the $L^2$ data setting.
By contrast, the assumptions we impose on our model do not guarantee
any exact conservation law at the $L^2$ level or at any other regularity
level. 

At the other end, both our use of the linear Schr\"odinger operator
and the boundedness condition (H1) are non-optimal, and we hope to relax
both of these restrictions in subsequent work. However, using these 
restrictions brings the major expository advantage that our model has a Galilean invariance, in the sense that a Galilean transformation yields
a problem that is in the same class, even though it is not exactly the same. This allows us to provide cleaner, shorter proofs for our results, and to keep the focus on the main ideas.

\subsection{ The main result}

Our main result asserts that global well-posedness holds for 
our problem for small $L^2$ data. In addition, our solutions not only satisfy uniform $L^2$, but also global space-time $L^6$ estimates,
as well as bilinear $L^2$ bounds, as follows:

\begin{theorem}\label{t:main}
Under the above assumptions (H1), (H2) and (H3) on the symbol of the cubic form $C$, small initial data 
\[
\|\du_0\|_{L^2} \leq \epsilon \ll 1,
\] 
yields a unique global solution $u$ for \eqref{nls}, which satisfies the following bounds:

\begin{enumerate}[label=(\roman*)]
\item Uniform $L^2$ bound:
\begin{equation}\label{main-L2}
\| u \|_{L^\infty_t L^2_x} \lesssim \epsilon.
\end{equation}

\item Strichartz bound:
\begin{equation}\label{main-Str}
\| u \|_{L^6_{t,x}} \lesssim \epsilon^\frac23.
\end{equation}

\item Bilinear Strichartz bound:
\begin{equation}\label{main-bi}
\| \partial_x (u \bar u(\cdot+x_0))\|_{L^2_t H_x^{-\frac12}} \lesssim \epsilon^2,
\qquad x_0 \in \R.
\end{equation}
\end{enumerate}
\end{theorem}

Here we note that in the case $x_0 = 0$ the last bound
gives 
\begin{equation}\label{main-bi-diag}
\| \partial_x |u|^2\|_{L^2_tH^{-\frac12}_x} \lesssim \epsilon^2,
\end{equation}
which is the more classical formulation of the bilinear $L^2$ bound. However, 
making this bound uniform  with respect to the  $x_0$ translation
captures the natural separate translation invariance of this 
bound, and is also quite useful in our proofs.

We also remark that all the bounds above are indeed Galilean invariant. As noted earlier, our main equation is not Galilean invariant, but the class of equations we are considering is. The estimates in the theorem do not represent the full strength of  what we actually prove, but are a merely a simple but relevant sample. Our actual proof yields stronger frequency envelope bounds associated to a decomposition of the solution $u$ on a unit frequency scale (rather than the more traditional dyadic decomposition); see Theorem~\ref{t:boot} in Section~\ref{s:boot}.

Applied to the model cubic NLS problem \eqref{nls3}, by scaling we have the following result which applies to the large data problem:

\begin{theorem}
Consider the defocusing 1-d cubic NLS problem \eqref{nls3}(+) with $L^2$ initial data $\du_0$.
Then the global solution $u$ satisfies the following bounds:

\begin{enumerate}[label=(\roman*)]
\item Uniform $L^2$ bound:
\begin{equation}\label{main-L2-model}
\| u \|_{L^\infty_t L^2_x} \lesssim \|\du_0\|_{L^2_x}.
\end{equation}

\item Strichartz bound:
\begin{equation}\label{main-Str-model}
\| u \|_{L^6_{t,x}} \lesssim  \|\du_0\|_{L^2_x}.
\end{equation}

\item Bilinear Strichartz bound:
\begin{equation}\label{main-bi-model}
\| \partial_x |u|^2\|_{L^2_t (\dot H_x ^{-\frac12} + c L^2_x) } \lesssim  \|\du_0\|_{L^2}^2, \qquad c = \| \du_0\|_{L^2}.
\end{equation}
\end{enumerate}
\end{theorem}

One may compare the above $L^6$ bound with the  Planchon-Vega estimate 
\eqref{PV-est}, see \cite{PV}, which applies only to $H^1$
solutions.
\bigskip

There are several ideas which play key roles in our analysis, all of which are used in a nonstandard 
fashion in the present work:
\medskip

\emph{ 1. Energy estimates via density flux identities.} This is 
a classical idea in pde's, and particularly in the study of conservation laws, namely that the density-flux identities 
play a more fundamental role than just energy identities. The new 
twist in our context is that this analysis is carried out in 
a nonlocal setting, where both the densities and the fluxes involve 
translation invariant multilinear forms.
\medskip

\emph{2. The use of energy corrections.}
This is an idea originally developed in the context of the so called 
I-method~\cite{I-method} or more precisely the second generation I-method \cite{I-method2}, whose aim was to construct more accurate almost conserved quantities. Here we implement this idea at the level of density-flux
identities, in a form closer to \cite{KT}. 
\medskip

\emph{3. Interaction Morawetz bounds.} These were originally developed in the context of the three-dimensional NLS problems by Colliander-Keel-Stafillani-Takaoka-Tao in \cite{MR2053757},
and have played a fundamental role in the study of many nonlinear Schr\"odinger flows, see e,g. \cite{MR2415387,MR2288737}, and also for one-dimensional quintic flows in the work of Dodson~\cite{MR3483476,MR3625190}. Our take on this is somewhat closer to the one-dimensional  approach of Planchon-Vega \cite{PV},
though recast in the setting and language of nonlocal multilinear 
forms.
\medskip

\emph{4. Tao's frequency envelope method.} This is used as 
a way to accurately track the evolution of the energy distribution across frequencies. Unlike the classical implementation relative to dyadic Littlewood-Paley decompositions, here we adapt and refine this notion 
for lattice decompositions instead. This is also very convenient as a bootstrap tool,  see e.g. Tao~\cite{Tao-WM}, \cite{Tao-BO} but 
with the added twist of also bootstrapping bilinear Strichartz bounds, as in the authors' paper \cite{IT-BO}.

\subsection{ An outline of the paper} 
In the next section we begin by setting up the notations
for function spaces and multilinear forms. More importantly, we also introduce our class of admissible frequency envelopes associated to lattice decompositions; this is based on the maximal function.

In Section~\ref{s:local} we carry our a preliminary step in the proof of our main result, namely we prove the small data local well-posedness result. This is independent 
of the global result, and uses a contraction argument
in a well chosen function space defined via a wave packet type decomposition.

The goal of Section~\ref{s:energy} is to recast energy identities for the mass and the momentum in density-flux form. We supplement this with two additional steps, 
where we first consider frequency localized mass and momentum  densities, and then we improve their accuracy by
adding a well chosen quartic correction.

In Section~\ref{s:Morawetz} we begin with the classical idea of interaction Morawetz identities for the linear
Schr\"odinger flow, and then we use our density-flux 
identities for the sharp frequency localized mass and momentum in order to obtain a set of refined interaction 
Morawetz identities for our problem. For clarity 
of exposition we consider separately the diagonal case, where the interaction of equal frequency components is considered, and the transversal case, which corresponds to separated frequency ranges.

The proof of our global result uses a complex bootstrap argument, involving both energy, Strichartz and bilinear
$L^2$ bounds in a frequency localized setting and based  on frequency envelopes. The bootstrap set-up is laid out in Section~\ref{s:boot}, which also contains a sharper, frequency envelope version of our result in Theorem~\ref{t:boot}. Our main estimates closing the bootstrap argument are carried out in Section~\ref{s:fe-bounds}, using the density-flux and interaction Morawetz identities previously obtained.

Finally, in the last section of the paper we return from
frequency localized bounds to global bounds, in order 
to complete the proof of our main global result.
\subsection{Acknowledgements} 
 The first author was supported by a Luce Professorship, by the Sloan Foundation, and by an NSF CAREER grant DMS-1845037. The second author was supported by the NSF grant DMS-2054975 as well as by a Simons Investigator grant from the Simons Foundation. Some of this work was carried out while both authors were 
 participating in the MSRI program ``Mathematical problems in fluid dynamics" during Spring 2021.

The authors also wish to thank the anonymous referee for 
the very careful reading of the manuscript and for the very useful corrections and suggestions.

\

\section{Notations and preliminaries}

\subsection{Lattice frequency decompositions}
For our analysis it will be convenient to localize functions 
in (spatial) frequency on the  unit scale. For this we consider
a partition of unity 
\[
1 = \sum_{k \in \Z} p_k(\xi),
\]
where $p_k$ are smooth bump functions localized in $[k-1,k+1]$.
Correspondingly, our solution $u$ will be decomposed as
\[
u = \sum_{k \in \Z} u_k, \qquad u_k = P_k u. 
\]

The main estimates we will establish for our solution $u$
will be linear and bilinear estimates for the functions $u_k$.

For a larger interval $A \subset \Z$, we denote
\[
u_A = \sum_{k \in A} u_k.
\]

\subsection{ Frequency envelopes} \label{s:fe}
This is a tool which allows us 
to more accurately track the distribution of energy at various frequencies for the solutions to nonlinear evolution equations. In the present paper,
they play a key bookkeeping role in the proof of the linear and bilinear bounds for our solutions in the context of a complex bootstrap argument. In brief, given some standard decomposition of, say, an
$L^2$ function 
\[
u = \sum u_k,
\]
a frequency envelope for $u$ is a sequence $\{c_k\}$ with the property that 
\[
\| u_k \|_{L^2} \lesssim c_k, \qquad \|c_k\|_{\ell^2} \approx \|u\|_{L^2}.
\]
In addition, one also limits how rapidly the sequence $\{c_k\}$ is 
allowed to vary. As  originally introduced  in work of Tao, see e.g. \cite{Tao-WM}, in the context of dyadic Littlewood-Paley decompositions,
 one assumes that the sequence $\{c_k\}$ is slowly varying, in the sense that
\[
\frac{c_j}{c_k} \leq 2^{\delta|k-j|}.
\]

Here we will instead work with  a uniform lattice decomposition on the unit frequency scale. This requires a major revision of the above notion of ``slowly varying", which turns out to be far too weak for our purposes. 

Instead we want to strengthen this property in order to say that $c \approx Mc$ (the maximal function):

\begin{definition}
A lattice frequency envelope $\{c_k\}$ is said to  
have the maximal property if 
\begin{equation}
Mc \leq C c    ,
\end{equation}
where $Mc$ represents the maximal function of $c$,
\[
(Mc)_k = \sup_{j \geq 0} \frac{1}{2j+1} \sum_{l=k-j}^{k+j} c_l.
\]
Here $C$ is a universal constant.
\end{definition}

Frequency envelopes that have 
this property will be called \emph{admissible}.
An important observation is that admissible envelopes can always be found:

\begin{lemma}
Any $\ell^2$ frequency envelope $c^0$
can be placed under a comparable maximal frequency envelope $c$, i.e.
\begin{equation}
c^0 \leq c, \qquad \| c\|_{\ell^2}    \approx \|c^0\|_{\ell^2}.
\end{equation}
\end{lemma}

\begin{proof}
We will use two properties of the maximal function:

(i) $\|Mf\|_{L^2} \leq  C \|f\|_{L^2}$

(ii) $M(f+g) \leq Mf + Mg$.

Given $c^0$, we define $c$ as 
\[
c = \sum_{k=0}^\infty (2C)^{-k} M^k c^0.
\]
By property (i), this series converges in $\ell^2$, with 
\[
\|c\|_{\ell^2} \leq 2 \|c^0\|_{\ell^2}.
\]

Then by property (ii) we have 
\[
Mc \leq 2C c.
\]
The proof is concluded.
\end{proof}

For an interval $A \subset \Z$ we denote
\[
c_A^2 = \sum_{k \in A} c_k^2.
\]
Also for a dyadic integer $n$ we set
\[
c_n^2 = \sum_{|k| \approx n} c_k^2.
\]

Also given a translation invariant function space $X$, we denote by $X_c$ the associated frequency envelope controlled norm
\begin{equation}
    \label{Xc}
\|u\|_{X_c} = \sup_k c_k^{-1} \|u_k\|_{X}.
\end{equation}

\subsection{Multilinear forms and symbols}\label{s:multi}

A key notion which is used throughout the paper is that 
of multilinear form. 
All our multilinear forms are invariant with respect to 
translations, and have as arguments either complex valued 
functions or their complex conjugates. 

For an integer $k \geq 2$, we will use 
translation invariant $k$-linear  forms 
\[
(\mathcal D(\R))^{k} \ni (u_1, \cdots, u_{k}) \to     L(u_1,\bu_2,\cdots) \in \mathcal D'(\R),
\]
where the nonconjugated and conjugated entries are alternating.

Such a form is uniquely described by its symbol $\ell(\xi_1,\xi_2, \cdots,\xi_{k})$
via
\[
\begin{aligned}
L(u_1,\bu_2,\cdots)(x) = (2\pi)^{-k} & 
\int e^{i(x-x_1)\xi_1} e^{-i(x-x_2)\xi_2}
\cdots 
\ell(\xi_1,\cdots,\xi_{k})
\\ & \qquad 
u_1(x_1) \bu_2(x_2) \cdots  
dx_1 \cdots dx_{k} d\xi_1\cdots d\xi_k,
\end{aligned}
\]
or equivalently on the Fourier side
\[
\mathcal F L(u_1,\bu_2,\cdots)(\xi)
= (2\pi)^{-\frac{k-1}2} \int_{D}
\ell(\xi_1,\cdots,\xi_{k})
\hat u_1(\xi_1) \bar{\hat u}_2(\xi_2) \cdots  
d\xi_1 \cdots d\xi_{k-1},
\]
where, with alternating signs, 
\[
D = \{ \xi = \xi_1-\xi_2 + \cdots \}.
\]

They can also be described via their kernel
\[
L(u_1,\bu_2,\cdots)(x) =  
\int K(x-x_1,\cdots,x-x_{k})
u_1(x_1) \bu_2(x_2) \cdots  
dx_1 \cdots dx_{k},
\]
where $K$ is defined in terms of the  
Fourier transform  of $\ell$
\[
K(x_1,x_2,\cdots,x_{k}) = 
(2\pi)^{-\frac{k}2} \hat \ell(-x_1,x_2,\cdots,(-1)^k x_{k}).
\]

All the symbols in this article will be 
assumed to be smooth, bounded and with bounded derivatives.

We remark that our notation is slightly nonstandard because of the alternation of complex conjugates, which is consistent with the set-up of this paper. Another important remark is that, for $k$-linear forms, the cases of odd $k$, respectively even $k$ play different roles here, as follows:

\medskip

i) The $2k+1$ multilinear forms will be thought of as functions, e.g. those which appear 
in some of our evolution equations.

\medskip

ii) The $2k$ multilinear forms will be thought of as densities, e.g. which appear 
in some of our density-flux pairs.

\medskip
Correspondingly,  
to each $2k$-linear form $L$ we will associate
a $2k$-linear functional $\bL$ defined by 
\[
\bL(u_1,\cdots,u_{2k}) = \int_\R L(u_1,\cdots,\bu_{2k})(x)\, dx,
\]
which takes real or complex values.
This may be alternatively expressed 
on the Fourier side as 
\[
\bL(u_1,\cdots,u_{2k}) = (2\pi)^{1-k} \int_{D}
\ell(\xi_1,\cdots,\xi_{2k})
\hat u_1(\xi_1) \bar{\hat u}_2(\xi_2) \cdots  
\bar{\hat u}_{2k}(\xi_{2k})d\xi_1 \cdots d\xi_{2k-1},
\]
where, with alternating signs, the diagonal $D_0$ is given by
\[
D_0 = \{ 0 = \xi_1-\xi_2 + \cdots \}.
\]
Note that in order to define the multilinear functional $\bL$ we only need to know the symbol $\ell$ on $D_0$. There will be however 
more than one possible smooth extension of 
$\ell$ outside $D_0$. This will play a role in our story later on.

\subsection{Cubic interactions in Schr\"odinger flows}
Given three input frequencies $\xi_1, \xi_2,\xi_3$ for 
our cubic nonlinearity, the output will be at frequency 
\[
\xi_4 = \xi_1-\xi_2+\xi_3.
\]
This relation can be described in a more symmetric fashion as 
\[
\Delta^4 \xi = 0, \qquad \Delta^4 \xi := \xi_1-\xi_2+\xi_3-\xi_4 .
\]
This is a resonant interaction if and only if we have a similar relation for the associated time frequencies, namely 
\[
\Delta^4 \xi^2 = 0, \qquad \Delta^4 \xi^2 := \xi_1^2-\xi_2^2+\xi_3^2-\xi_4^2 .
\]
Hence, we define the resonant set in a symmetric fashion as 
\[
\calR := \{ \Delta^4 \xi = 0, \ \Delta^4 \xi^2 = 0\}.
\]
It is easily seen that this set may be characterized as
\[
\calR = \{ \{\xi_1,\xi_3\} = \{\xi_2,\xi_4\}\}.
\]

\subsection{The Galilean symmetry}
Here we investigate how the equation \eqref{nls}
changes if we apply a Galilean transformation. 
In particular, we will justify our claim in the introduction that the transformed equation is of the same type.

We first recall the linear case. Suppose $u$ 
solves the linear Schr\"odinger equation
\[
(i \partial_t + \partial_x^2)   u = f, \qquad u(0) = \du_0.
\]
Given a frequency $k$, its Galilean transform 
$v$ is defined by
\[
v(t,x) := e^{-i(kx+k^2 t)} u(t,x+2kt),
\]
and solves the linear Schr\"odinger equation
\[
(i \partial_t + \partial_x^2)   v = g, \qquad v(0) = \dv_0,
\]
where
\[
 \dv_0(x) = e^{-ik x} u_{0}, \qquad g(t,x) = e^{-i(kx+k^2 t)} f(t,x+2kt).
\]

Now suppose that $u$ solves \eqref{nls}.
Then the above computation shows that $v$ will solve a similar equation,
\[
(i \partial_t + \partial_x^2)   v = \tilde C(v,\bv,v),
\]
where
\[
\tilde C(v,\bv,v) = e^{-ikx} C(v e^{ikx}, \overline{v e^{ikx}}, v e^{ikx}).
\]
This allows us to compute the symbol of $\tilde C$
as 
\[
\tilde c(\xi_1,\xi_2,\xi_3) = c(\xi_1-k,\xi_2-k,\xi_3-k).
\]
This translated symbol is easily seen to have exactly the same properties as $c$.

\section{Local well-posedness }
\label{s:local}

Before approaching the global problem, an initial 
step is to establish local in time well-posedness.
Since we only assume boundedness and 
smoothness on the symbol $C$,
this is not an entirely straightforward matter.
Our main result can be summarily stated as follows:

\begin{theorem}\label{t:local}
The evolution \eqref{nls} is locally well-posed for small data in $L^2$.
\end{theorem}

Here we need to clarify the meaning of well-posedness.
For this problem, we will establish a semilinear type of well-posedness result. Precisely, for each initial data $\du_0$ which is small in $L^2$ a unique solution exists in $C([0,1];L^2)$, with Lipschitz dependence on the initial data.

However, as it is often the case in the dispersive realm, we will not try to prove unconditional uniqueness, and contend ourselves with having both existence and uniqueness of solutions in a ball in a restricted space $X \subset C([0,1];L^2)$. 

A natural follow-up question here would be whether 
the same result holds for large data in our context.
The answer is indeed affirmative; however, in this 
article we have chosen to only consider small data
because this is all we need on one hand, and a large
data result would require a more complex choice of the space $X$ mentioned above, as well as a correspondingly 
more complex proof, on the other hand.

Another related question is whether a standard scaling argument could be used here.  The scaling transformation would be the standard one for the cubic NLS problem,
\[
u_{\lambda}(t,x) = \lambda u(\lambda^2 t,\lambda x).
\]
For the initial data this corresponds to 
\[
\du_{0\lambda}(x) = \lambda \du_0(\lambda x).
\]

It is then easy to see that $u_\lambda$ solves an equation of the same type as \eqref{nls}, but with the rescaled symbol
\[
c_\lambda(\xi_1,\xi_2,\xi_3) = c(\xi_1/\lambda, \xi_2/\lambda,\xi_3/\lambda).
\]
This satisfies the bound \eqref{c-smooth} uniformly
only for $\lambda \geq 1$, so it cannot be used to reduce the large data problem to the small data problem.
However, it can be used to obtain better life-span bounds
for small data:

\begin{corollary}
Assume that the initial data $\du_0$ for \eqref{nls}
satisfies $\|\du_0\|_{L^2} \leq \epsilon$. Then the solution $u$ exists on $[0, T_\epsilon]$ with $T_\epsilon :=c \epsilon^{-2}$, with similar bounds.
\end{corollary}

The rest of this section is devoted to the proof of Theorem~\ref{t:local}.  The first step in our proof is to construct a suitable function space $X$ where we seek the solutions.

Given a function $u$ in $[0,1] \times \R$, we start with a decomposition $u = \sum_{k \in \Z} u_k$ on the unit frequency scale, 
and then a partition of unity in the physical space,
also on the unit scale,
\[
1 = \sum_{j \in \Z} \chi_j(x). 
\]
Finally, we define the norm of the space $X$ for solutions
\begin{equation}\label{def-X}
\| u \|_{X}^2 = \sum_{k \in \Z} \| u_k\|_{X_k}^2, \qquad 
\|u_k\|_{X_k}^2= \sum_{j \in \Z} \| \chi_j (t, x-2t k) u_k \|_{L^\infty_t L^2_x}^2.      
\end{equation}
Here the second argument of $\chi_j$ is consistent with the group velocity of frequency $k$ waves. Indeed, 
if $u$ were an $L^2$ solution to the homogeneous Schr\"odinger equation then this would be nothing but a wave packet decomposition of $u$ on the unit time scale. It is easily seen that we have the embedding 
\[
X \subset L^\infty_t L^2_x.
\]
\begin{remark} \label{r:X}
Due to the unit frequency localization of $u_k$ and Bernstein's inequality,
we may freely replace the $L^\infty_t L^2_x$ norm in 
\eqref{def-X} by $L^\infty_{t,x}$.
\end{remark}

Correspondingly, we define a similar space $Y$ for the source term in a linear Schr\"odinger equation, namely
\begin{equation}\label{def-Y}
\| f \|_{Y}^2 = \sum_{k \in \Z} \| f_k\|_{Y_k}^2, \qquad 
\| f_k\|_{Y_k}^2 = \sum_{j\in Z}  \| \chi_j (t, x-2t k) f_k \|_{L^1_t L^2_x}^2  ,    
\end{equation}
so that we have the duality relation 
\[
X = Y^*,
\]
with equivalent norms.

Then for the small data local well-posedness result in $X$
it suffices to establish the following two properties. The first is a linear mapping property:

\begin{lemma}\label{l:lwp-lin}
The solution to the linear Schr\"odinger equation 
\begin{equation}
(i \partial_t + \partial_x^2)   u = f, \qquad u(0) = \du_0
\end{equation}
in the time interval $[0,1]$ satisfies 
\begin{equation}
\| u\|_{X}    \lesssim \|\du_0\|_{L^2} + \|f\|_{Y}.
\end{equation}
\end{lemma}

The second is an estimate for the nonlinearity: 

\begin{lemma}\label{l:lwp-cubic}
For the cubic nonlinearity $C$ we have the bound
\begin{equation}\label{C-Y}
\| C(u,\bar u,u) \|_{Y} \lesssim \|u\|_{X}^3  .  
\end{equation}
\end{lemma}

Once we have these two lemmas, the proof of the local well-posedness result follows in a standard manner using the contraction principle in a small ball in $X$. However, for later use we also need to have 
a more precise, frequency envelope version of Theorem~\ref{t:local}.
This is as follows:

\begin{theorem}\label{t:local-fe}
For each small initial data 
\[
\|\du_0\|_{L^2} \leq \epsilon \ll 1
\]
there exists a unique solution $u$ to \eqref{nls} which is small in $X$.
In addition, suppose $c_k$ is an $\ell^2$ normalized  admissible frequency envelope so that 
\[
\|\du_0\|_{L^2_c} \lesssim \epsilon.
\]
Then the solution $u$ satisfies
\begin{equation}
\| u \|_{X_c} \lesssim \epsilon. 
\end{equation}
\end{theorem}

This requires stronger, frequency envelope versions of 
Lemmas~\ref{l:lwp-lin}, \ref{l:lwp-cubic}:

\begin{lemma}\label{l:lwp-lin-fe}
The solution to the linear Schr\"odinger equation 
\begin{equation}
(i \partial_t + \partial_x^2)   u = f, \qquad u(0) = \du_0
\end{equation}
in the time interval $[0,1]$ satisfies 
\begin{equation}
\| u\|_{X_c}    \lesssim \|\du_0\|_{L^2_c} + \|f\|_{Y_c}.
\end{equation}
\end{lemma}

The second is an estimate for the nonlinearity: 

\begin{lemma}\label{l:lwp-cubic-fe}
 Let $c_k$ be an $\ell^2$ normalized  admissible frequency envelope.
Then for the cubic nonlinearity $C$ we have the bound
\begin{equation}\label{C-Y-fe}
\| C(u,\bar u,u) \|_{Y_c} \lesssim \|u\|_{X_c}^3  .  
\end{equation}
\end{lemma}

\begin{proof}[Proof of Lemmas~\ref{l:lwp-lin},
~\ref{l:lwp-lin-fe}]
We can freely localize on the unit scale in frequency,
and reduce the problem to the frequency localized estimate
\begin{equation}\label{uk-bd}
\|u_k\|_{X_k}    \lesssim \|\du_{0k}\|_{L^2} + \|f_k\|_{Y_k}.
\end{equation}
We can further reduce the problem by applying a Galilean transformation, by setting 
\[
v(t,x) = e^{-i(kx+k^2 t)} u_k(t,x-2kt), \quad \dv_0(x) = e^{-ik x} \du_{0k}, \quad g(t,x) = e^{-i(kx+k^2 t)} f_k(t,x-2kt).
\]
Here the functions $v,\dv_0,g$ are now localized at frequency $0$
and solve
\[
(i \partial_t + \partial_x^2)   v = g, \qquad v(0) = \dv_0,
\]
whereas the bound \eqref{uk-bd} reduces to 
\begin{equation}\label{v-bd}
  \|v\|_{X_0}    \lesssim \|\dv_{0}\|_{L^2} + \|g\|_{Y_0}. 
\end{equation}
Inserting a harmless frequency localization $P_0$, we represent $v$ as 
\[
v(t) = e^{it \partial_x^2} P_0 \dv_0 -i  \int_0^t e^{i(t-s) \partial_x^2} P_0 g\, ds .
\]
Here by a slight abuse a notation we allow $P_0$ to have slightly larger support.
Finally, we localize spatially at both ends,
\[
\chi_j v(t) = \sum_{l \in \Z} \left( \chi_j e^{it \partial_x^2} P_0 \chi_l \dv_0 +  \int_0^t \chi_j e^{i(t-s) \partial_x^2}  P_0 \chi_l g \, ds\right).
\]
Here the kernels for $e^{it \partial_x^2} P_0$ are uniformly 
Schwartz for $t \in [0,1]$, so we get an $L^2$ bound with off-diagonal decay,
\[
\| \chi_j e^{it \partial_x^2} P_0 \chi_l\|_{L^2 \to L^2}
\lesssim \la j-l\ra^{-N}.
\]
This implies that 
\[
\| \chi_j v\|_{L^\infty_t L^2_x} \lesssim 
\sum_{l \in \Z} \la j-l\ra^{-N} \left( \|\chi_l \dv_{0}\|_{L^2_{x}} + \|\chi_l g\|_{L^1_t L^2_x}\right),
\]
which in view of the off-diagonal decay implies the bound \eqref{v-bd}. 

\end{proof}

\begin{proof}[Proof of Lemmas~\ref{l:lwp-cubic},
\ref{l:lwp-cubic-fe}]
Here the second lemma implies the first.
We need to prove the estimate
\[
\| P_{k} C(u,\bu,u)\|_{Y_{k}} \lesssim c_k \| u\|_{X_c}^3. 
\]
By duality, this reduces to 
the integral bound 
\begin{equation}\label{I-est}
\left| I \right| \lesssim c_k \| u\|_{X_c}^3 \|v_k\|_{X_k}, \qquad     I = \int C(u,\bu,u) \bv_k \, dx dt.
\end{equation}
Without any restriction in generality we may assume 
that 
\[
\|u\|_{X_c} = 1, \qquad \| v_k\|_{X_k}=1.
\]
We use the unit scale frequency decomposition to separate the above integral as 
\[
I = \sum_{k_1-k_2+k_3=k}
\int C_{k_1k_2k_3}(u_{k_1},\bu_{k_2},u_{k_3}) \bv_{k} \, dx dt,
\]
where we have also localized the kernel of $C$ near frequencies 
$k_1$, $k_2$, $k_3$ on the unit scale. The symbol of $C_{k_1k_2k_3}$ is smooth and bounded on the unit scale, so
the above summands are essentially like products, and may be indeed thought of as products via separation of variables. 
For bilinear products we have the estimate
\begin{equation}\label{separate-bi}
\|    u_{k_1} v_{k_2}\|_{L^2_{x,t}} \lesssim \frac{1}{\la k_1-k_2\ra^\frac12} \|   u_{k_1}\|_{X_{k_1}} \|v_{k_2}\|_{X_{k_2}}.
\end{equation}
This is obtained simply by examining the intersection of the supports of the bump functions traveling with speeds $2k_1$, respectively $2k_2$. 

Denoting   $\delta k^{hi} = \max |k_i - k_j|$, the relation $k_1+k_3 = k_2+k$ insures 
that we can group the four frequencies into two pairs at distance $\delta k^{hi}$. Then, using twice the above bilinear estimate, we have
\[
|I| \lesssim \sum_{k_1-k_2+k_3-k= 0} 
\frac{1}{\la \delta k^{hi}\ra} c_{k_1} c_{k_2} c_{k_3} .
\]
Let $n$ represent the dyadic size of 
$\delta k^{hi}$.
Without loss of generality, by relabeling, 
suppose that 
\[
|k-k_3| \approx |k-k_2| \approx n, \qquad 
|k - k_1| \lesssim n.
\]
 Then using the Cauchy-Schwartz inequality for the pair 
 $(k_2,k_3)$ for fixed $k_1$ we estimate
\[
|I| \lesssim \sum_{n} \frac{1}{n}\sum_{|k_1-k| \lesssim n} 
c_{k_1} c_{n}^2, \qquad c_n^2:= \sum_{|j-k|\approx n} c_j^2.
\]
Now we use the maximal function inequality 
for $c$, which gives 
\[
\frac{1}{n} \sum_{|k_1-k| \lesssim n} 
c_{k_1} \lesssim c_{k}.
\]
We obtain
\[
|I| \lesssim c_{k} \sum_{n} c_n^2 \approx c_k. 
\]
 Thus \eqref{I-est} is proved.
\end{proof}

For later use, we note that the frequency envelope bounds for $u$ together with the bilinear $L^2$ bound \eqref{separate-bi} imply the following 

\begin{corollary}\label{c:local-fe}
Let $u$ be a solution for \eqref{nls} in $[0,1]$
as in Theorem~\ref{t:local-fe}. Then the following bounds hold:
\begin{equation}
\| u_k \|_{L^6_{t,x}} \lesssim \epsilon c_k,    
\end{equation}

\begin{equation}
\| \partial_x (u_{k_1} \bu_{k_2}(\cdot +x_0))\|_{L^2_{t,x}}
\lesssim \epsilon^2 \la k_1-k_2 \ra^\frac12 c_{k_1}c_{k_2}.
\end{equation}

\end{corollary}

\section{Energy estimates and conservation laws}
\label{s:energy}

\subsection{Conservation laws for the linear problem}
We begin our discussion with the linear Schr\"odinger equation
\begin{equation}
i u_t + u_{xx} = 0, \qquad u(0) = \du_0.    
\end{equation}
For this we consider the following three conserved quantities, the mass
\[
\bM(u) = \int |u|^2 \,dx,
\]
the momentum 
\[
\bP(u) = 2 \int \Im (\bar u \partial_x u) \,dx,
\]
as well as the energy
\[
\bE(u) = 4 \int |\partial_x u|^2\, dx. 
\]

To these quantities we associate corresponding densities
\[
M(u) = |u|^2, \qquad P(u) = i ( \bar u \partial_x u - u \partial_x \bar u), \qquad E(u) = - \bar u \partial_x^2 u   + 
2  |\partial_x u|^2 -  u \partial_x^2 \bar u.
\]
The choice of densities here is not entirely straightforward. Symmetry is clearly a criteria, but further motivation is provided by the conservation law computation,
\begin{equation}\label{df-lin}
\partial_t M(u) = \partial_x P(u), \qquad \partial_t P(u) = \partial_x E(u).
\end{equation}
The symbols of these densities viewed as bilinear forms are
\[
m(\xi,\eta) = 1, \qquad p(\xi,\eta) = -(\xi+\eta), \qquad e(\xi,\eta) = (\xi+\eta)^2.
\]

More generally, we can start start with a symbol $a(\xi,\eta)$ which is symmetric, in the sense that
\[
a(\eta,\xi) = \ol{a(\xi,\eta)},
\]
and then define an associated weighted mass density by
\[
M_a(u) = A(u,\bar u).
\]
We also define corresponding momentum and energy symbols $p_a$ and $e_a$ by 
\[
p_a(\xi,\eta) = -(\xi+\eta)a(\xi,\eta), \qquad e_a(\xi,\eta) = (\xi+\eta)^2a(\xi,\eta).
\]
Then a direct computation yields the density flux relations
\[
\frac{d}{dt} M_a(u,\bar u) = \partial_x P_a(u,\bar u), \qquad \frac{d}{dt} P_a(u,\bar u) = \partial_x E_a(u,\bar u).
\]

\subsection{Nonlinear density flux identities for the mass and momentum} 
Here we develop the counterpart of the linear analysis above for the nonlinear problem \eqref{nls}.

\subsubsection{The modified mass} 
To motivate what follows, we begin with a simpler computation for the $L^2$ norm of a solution $u$ of \eqref{nls}: 
\[
\frac{d}{dt} \| u\|_{L^2}^2 = \int - i C(u,\bar u,u) \cdot \bar u  + i u \cdot \ol{C(u,\bar u,u)} \, dx
:=   \int C^4_m(u,\bar u, u, \bar u) \, dx.
\]
A-priori the symbol of the quartic form $C^4_m$, defined on the diagonal $\Delta^4 \xi = 0$,
is given by
\[
c^4_m(\xi_1,\xi_2,\xi_3,\xi_4) = - i c(\xi_1,\xi_2,\xi_3) + i \bar c(\xi_2,\xi_3,\xi_4).
\]
However, we can further symmetrize and replace it by 
\[
c^4_m(\xi_1,\xi_2,\xi_3,\xi_4) = \frac{i}2 \left( -  c(\xi_1,\xi_2,\xi_3) - c(\xi_1,\xi_4,\xi_3)+ \bar c(\xi_2,\xi_3,\xi_4)+ 
\bar c(\xi_2,\xi_1,\xi_4) \right).
\]
In particular we are interested in the behavior of $c^4_m(\xi_1,\xi_2,\xi_3,\xi_4)$ on the resonant set
\[
\calR = \{(\xi_1, \xi_2, \xi_3, \xi_4)\in \mathbb{R}^4 \, / \,\Delta^4 \xi = 0, \,  \Delta^4 \xi^2 = 0\} = \{ \{ \xi_1,\xi_3\} = \{\xi_2,\xi_4\} \}.
\]
On this set we compute
\[
\begin{split}
c^4_m(\xi_1,\xi_1,\xi_3,\xi_3) = & \ \frac{i}2 \left( -  c(\xi_1,\xi_1,\xi_3) - c(\xi_1,\xi_3,\xi_3)+ \bar c(\xi_1,\xi_3,\xi_3)+ 
\bar c(\xi_1,\xi_1,\xi_3) \right)
\\
= & \ \Im( c(\xi_1,\xi_1,\xi_3) + c(\xi_1,\xi_3,\xi_3)).
\end{split}
\]
Then we observe that our (H2) assumption on $C$ shows that this expression vanishes. One might wonder here if we could not weaken this assumption by
requiring that the sum of the two terms is zero, rather than each of
them separately.  This would indeed be the case if all we were interested in 
is the almost conservation of mass. However, we will later add localization weights which
will act differently on the two terms.

The fact that $c^4_m$ vanishes on the resonant set $\calR$ implies (see Lemma~\ref{l:division} below) that we can smoothly divide 
\[
b^4_m(\xi_1,\xi_2,\xi_3,\xi_4) = - \frac{i c^4_m(\xi_1,\xi_2,\xi_3,\xi_4)}{\Delta^4 \xi^2}
\]
on $\Delta^4 \xi = 0$.
We now use $\bB^4_m$ as an energy correction. Then we obtain the modified
energy relation
\begin{equation}\label{modified-mass}
\frac{d}{dt} (\| u\|_{L^2}^2+ \bB^4_m(u,\bar u, u,\bar u) )  = \bR^6_m(u,\bar u,  u,\bar u, u,\bar u) ,
\end{equation}
where $\bR^6_m$ is a symmetric $6$-linear form. 
Here the left hand side may be viewed as a modified
energy, while the right hand side
can potentially be estimated using the $L^6_{t,x}$ norm of $u$.

\subsubsection{ The modified mass and momentum density-flux pairs}

The key idea here is that, corresponding to the above modified mass, we also want to write a conservation law for an associated mass density
\begin{equation}\label{m-sharp}
\ms(u) = M(u) + B^4_m(u,\bar u, u,\bar u).
\end{equation}
However, when doing this, we remark that the symbol of $B^4_m$ was previously 
defined only on the diagonal $\Delta^4 \xi = 0$, whereas in order for the above expression to be well defined we need to extend it everywhere. For the purpose of this computation we simply assume that we have chosen some smooth extension.
A more careful choice will be considered later in Lemma~\ref{l:division}.

Now we compute 
\[
\partial_t \ms(u) = \partial_x P(u) + C^4_m(u,\bar u, u,\bar u) + i ( \Delta^4 \xi^2 B^4_m)(u,\bar u, u,\bar u) + 
R^6_m(u,\bar u,  u,\bar u, u,\bar u) .
\]
By the choice of $ B^4_m$, the symbol of the quartic term above $c^4_m +  i  \Delta^4 \xi^2 b^4_m$ vanishes 
on the diagonal $\{\Delta^4 \xi = 0\}$, therefore we can express it smoothly in the form
\begin{equation}\label{choose-R4m}
c^4_m +  i  \Delta^4 \xi^2 \,  b^4_m = i \Delta^4 \xi\, r^4_m .
\end{equation}
Hence the above relation can be written in the better form
\begin{equation}\label{dens-flux-m}
\partial_t \ms(u) = \partial_x (P(u) + R^4_m(u,\bar u, u,\bar u)) +
R^6_m(u,\bar u,  u,\bar u, u,\bar u) .
\end{equation}

One may view here the relation \eqref{choose-R4m} as a division problem,
where $c^4_m$ vanishes on the resonant set $\calR$. The symbols $b^4_m$ and $r^4_m$ are not uniquely determined by the relation \eqref{choose-R4m}, as we can change them by 
\[
b_m^4 \to b_m^4 + q \Delta^4 \xi, \qquad r_m^4 \to r_m^4 + q \Delta^4 \xi^2,
\]
for any smooth $q$. However, this is the only 
ambiguity. In particular $r_m^4$ is uniquely determined on the 
set $\Delta^4 \xi^2 = 0$, while $b_m^4$ is uniquely determined 
on the set $\Delta^4 \xi = 0$. 

\bigskip

One could carry out a similar computation for the momentum,
where the starting point is the relation
\[
\partial_t P(u) = \partial_x E(u) + C^4_p(u,\bar u, u, \bar u).
\]
Precisely, the symbol of $C^4_p$ is initially given by
\[
c^4_p(\xi_1,\xi_2,\xi_3,\xi_4) =  i (\xi_1-\xi_2+\xi_3+\xi_4) c(\xi_1,\xi_2,\xi_3) - i 
(\xi_1+\xi_2-\xi_3+\xi_4)\bar c(\xi_2,\xi_3,\xi_4).
\]
However, we can further symmetrize it exactly as in the case of  $C_m^4$. 
Then it also vanishes on the resonant set $\calR$, so it admits a 
(nonunique) representation of the form
\begin{equation}\label{choose-R4p}
c^4_p +  i  \Delta^4 \xi^2 b^4_p = i \Delta^4 \xi r^4_p. 
\end{equation}

Hence, as in the case of the mass, we define a quartic correction for the momentum density
\[
\ps(u) = P(u) + B^4_p(u,\bar u, u,\bar u).
\]
This satisfies a conservation law of the form
\begin{equation}\label{dens-flux-p}
\partial_t \ps(u) = \partial_x (E(u) + R^4_p(u,\bar u, u,\bar u)) +
R^6_p(u,\bar u,  u,\bar u, u,\bar u) .
\end{equation}

\subsection{ The choice for the density-flux corrections}

Here we consider the division problem  in \eqref{choose-R4m}, and ask
what should be a good balance between the symbols $B_m^4$ and $R_m^4$.
We recall that $b_m^4$ is uniquely determined on the diagonal $\Delta^4 \xi = 0$,
but we can choose it freely away from the diagonal.

To move away from the diagonal, it is useful to do it in a Galilean invariant fashion. The expression $\Delta^4 \xi^2$
is not Galilean invariant, but we do have a suitable replacement,
namely the expression
\[
\tDelta^4 \xi^2 := \Delta^4 \xi^2 - 2 \xi_{avg} \Delta^4 \xi 
= \frac12 ((\xi_1-\xi_3)^2 - (\xi_2-\xi_4)^2).
\]
This is easily seen to be invariant with respect to translations.
To measure the size of both $\Delta^4 \xi$ and $\tDelta^4 \xi^2$ we introduce two parameters,
\begin{equation}\label{xihm}
\begin{aligned}
\xih = & \ \max  \{|\xi_1 - \xi_{2}|+|\xi_3-\xi_4|, |\xi_1 - \xi_{4}|+|\xi_3-\xi_2| \},
\\
\xim =& \  \min  \{|\xi_1 - \xi_{2}|+|\xi_3-\xi_4|, |\xi_1 - \xi_{4}|+|\xi_3-\xi_2| \},
\end{aligned}
\end{equation}
where $\xih$ measures the diameter of the full set of $\xi$'s whereas
$\xim$ measures the distance of the sets $\{\xi_1,\xi_3\}$ and $\{\xi_2,\xi_4\}$.
With these notations, we have bounds from above as follows:
\begin{equation}\label{d4-xi-above}
|\Delta^4 \xi| \lesssim \xim, \qquad   |\tDelta^4 \xi^2| \lesssim \xih \xim.
\end{equation}
We will think of the symbol $\Delta^4 \xi$ as being elliptic where approximate equality holds in the first relation,  and of $\tDelta^4 \xi^2$ as being elliptic  where approximate equality holds in the second relation. Based on this, we will decompose the phase space into three overlapping regions 
which can be separated using cutoff functions which are smooth on the unit scale:
\begin{enumerate}[label=\roman*)]
\item The full division region,
\[
\Omega_1 = \{ \xim \lesssim 1 \},
\]
which represents a full unit size neighbourhood of the resonant set $\calR$.
\item The region 
\[
\Omega_2  = \{ 1+ |\Delta^4 \xi| \ll \xim \}, 
\]
where $\tDelta^4 \xi^2$  must be elliptic, $|\tDelta^4 \xi^2| \approx \xih \xim$,
and thus we will favor division by the symbol $\tDelta^4 \xi^2$.

\item The region
\[
\Omega_3 =  \{ 1 \ll \xim \lesssim |\Delta^4 \xi|  \},
\]
we will instead divide by $\Delta^4 \xi$; this 
is compensated by the relatively small size of this region.
\end{enumerate}

This decomposition leads us to the following division lemma:

\begin{lemma}\label{l:division}
Let $c^4$ be a bounded symbol which is smooth on the unit scale,
and which vanishes on $\calR$. Then it admits a representation
\begin{equation}
c^4 =  \Delta^4 \xi\, \tilde r^4 - \tDelta^4 \xi^2 \, b^4,      
\end{equation}
where $\tilde r^4$ and $b^4$ are also smooth on the unit scale, with 
the following properties:

\begin{enumerate}[label = \roman*)]
    \item Size 
\begin{equation} \label{rb4}
\begin{aligned}
|\partial^\alpha \tilde r^4|  \lesssim & \  \frac{1}{\la \xim \ra} , 
\\
|\partial^\alpha b^4|  \lesssim & \  \frac{1}{\la\xih\ra \la\xim\ra}.
\end{aligned}   
\end{equation}    
\item Support:
$b^4$ is supported in $\Omega_1 \cup \Omega_2$ and $\tilde r^4$ is supported in $\Omega_1 \cup \Omega_3$.
\end{enumerate}

\end{lemma}
Here and later in the paper by ``smooth on the unit scale" we mean that the above functions and all their derivatives are bounded, with bounds as in \eqref{rb4}, and where the implicit constant is allowed to depend on $\alpha$, but not on anything else. As usual, only finitely many derivatives are needed on our analysis, but we do  not take the extra step of determining how many.

\medskip

To return to $\Delta^4\xi$, we have the following straightforward 
observation:

\begin{remark}\label{r:division}
Later we will need similar decompositions but 
with $\tDelta^4 \xi^2$ replaced by $\Delta^4 \xi^2$,
\[
c^4 =  \Delta^4 \xi r^4 - \Delta^4 \xi^2 b^4    .
\]
This is easily done via the substitution
\[
r^4 = \tilde r^4 + 2 \xi_{avg} b^4.
\]
But in doing this, we loose the above bound for $\tilde r^4$ unless $|\xi_{avg}| \lesssim \xih$. Precisely,
we obtain instead
\begin{equation}\label{r4-bd}
  |\partial^\alpha r^4|  \lesssim  \frac{1}{\la \xim \ra}
  \left( 1+ \frac{|\xi_{avg}|}{\la \xih \ra}\right).
\end{equation}
\end{remark}

\begin{proof}
Using a partition of unity which is smooth on the unit scale, we can reduce the problem to the case 
when $c^4$ is supported in exactly one of the regions $\Omega_1$, $\Omega_2$ and $\Omega_3$. 
We consider each of these cases separately.
\medskip

\emph{i) $c^4$ is supported in $\Omega_1$.} To simplify notations here 
we introduce new linear coordinates $(\eta_1,\eta_2,\eta_3,\eta_4)$ where
\[
\eta_1 = \Delta^4 \xi, \quad \eta_2 = \xi_1+\xi_2-\xi_3-\xi_4, \quad \eta_3 = \xi_1-\xi_2-\xi_3+\xi_4, \qquad 2 \eta_2 \eta_3 = \tDelta^4 \xi^2.
\]
For $\eta_4$ we can choose in a symmetric fashion
\[
\eta_4 = \xi_1+\xi_2+\xi_3+\xi_4,
\]
though this does not play any role in the sequel.

In these coordinates we have 
\[
\Omega_1 = \left\{ |\eta_1|+ \min\{|\eta_2|, |\eta_3|\} \lesssim 1\right\} = \Omega_{11} \cup \Omega_{12} \cup \Omega_{13},
\]
where 
\[
\begin{aligned}
& \Omega_{11} = \left\{|\eta_1| + |\eta_2|+ |\eta_3| \lesssim 1\right\},  \qquad \Omega_{12} = \left\{ |\eta_1|+ |\eta_2| \lesssim 1 \lesssim |\eta_3| \right\},
 \\ & \qquad \qquad \Omega_{13} = \left\{|\eta_1|+ |\eta_3| \lesssim 1 \lesssim |\eta_2| \right\}.
 \end{aligned}
\]
Using another partition of unity which is smooth on the unit scale, the problem reduces to separately considering the case when $c^4$ is supported in each 
of these three sets.

Within the set $\Omega_{12}$ we have $c^4(0,0,\eta_3,\eta_4) = 0$ therefore we can easily 
represent 
\[
c^4(\eta_1,\eta_2,\eta_3,\eta_4) = (c^4(\eta_1,\eta_2,\eta_3,\eta_4)- 
c^4(0,\eta_2,\eta_3,\eta_4)) +
(c^4(0,\eta_2,\eta_3,\eta_4) - 
c^4(0,0,\eta_3,\eta_4)),
\]
where the first difference may be smoothly divided 
by $\eta_1$ and the second by $\eta_2$, with the quotients contributing to $\tilde r^4$, respectively $b^4$.
The set $\Omega_{13}$ can be dealt with in a similar fashion.

It remains to consider $\Omega_{11}$, where 
we know that $c^4 = 0$ in $\eta_1 = \eta_2 \eta_3 = 0$. Here we write
\[
c^4(\eta_1,\eta_2,\eta_3,\eta_4) = (c^4(\eta_1,\eta_2,\eta_3,\eta_4)- 
c^4(0,\eta_2,\eta_3,\eta_4)) +
c^4(0,\eta_2,\eta_3,\eta_4).
\]
Now the first difference can be smoothly divided by $\eta_1$, while the last term can be successively and smoothly divided by $\eta_2$ and $\eta_3$.

\medskip

\emph{ii) $c^4$ is supported in $\Omega_2$.} Here we set 
\[
b^4 = \frac{c^4}{\tDelta^4 \xi^2}, \qquad \tilde r^4 = 0,
\]
and we observe that 
\[
\left|\partial^\alpha\frac{1}{\tDelta^4 \xi^2}\right| \lesssim \frac{1}{\xim \xih}.
\]

\medskip

\emph{iii) $c^4$ is supported in $\Omega_3$.}
Here we set 
\[
b^4 = 0, \qquad \tilde r^4 = -\frac{c^4}{\Delta^4 \xi},
\]
and we observe that 
\[
\left|\partial^\alpha\frac{1}{\Delta^4 \xi}\right| \lesssim \frac{1}{\xim}.
\]

\end{proof}

\subsection{The Galilean invariance}
While the assumptions (H1-3) on the cubic nonlinearity $C$ are Galilean invariant, our density-flux identities are not. In order to rectify that, suppose heuristically that we are looking at linear waves concentrated around a frequency $\xi_0$. This corresponds to a linear group velocity of $2 \xi_0$, so in the density-flux identities
it would be natural to replace the operator $\partial_t$ by 
$\partial_t + 2 \xi_0 \partial_x$. At the linear level,
this is done by recentering the energy and momentum densities
at $\xi_0$, 
\begin{equation}
\begin{aligned}
p_{\xi_0}(\xi_1,\xi_2) = & \ - \xi_1 -\xi_2 + 2 \xi_0 = p + 2 \xi_0 m,
\\
e_{\xi_0}(\xi_1,\xi_2) = & \ (\xi_1+\xi_2 - 2 \xi_0)^2 = 
e + 4 \xi_0 p + 4 \xi_0^2 m.
\end{aligned}
\end{equation}
Then the density-flux identities \eqref{df-lin} become
\begin{equation}\label{df-lin-G}
(\partial_t  + 2 \xi_0 \partial_x) M(u) = \partial_x P_{\xi_0}(u), \qquad (\partial_t  + 2 \xi_0 \partial_x) P_{\xi_0}(u) = \partial_x E_{\xi_0}(u).
\end{equation}

Next we consider the nonlinear setting.
There $\ms$ is the same as before, but $\ps_{\xi_0}$ is 
\[
\ps_{\xi_0} = \ps - 2 \xi_0 \ms = P_{\xi_0} + B^4_{p,\xi_0} ,
\]
where the symbol for $B^4_{p,\xi_0}$ is given by 
\begin{equation}
 b^4_{p,\xi_0} =  b^4_{p} + 2\xi_0 b^4_{m}.
\end{equation}

Then our density-flux identities have the form
\begin{equation}\label{dens-flux-m-G}
(\partial_t  + 2 \xi_0 \partial_x) \ms(u) = \partial_x (p_{\xi_0}(u) + R^4_{m,\xi_0}(u,\bar u, u,\bar u)) +
R^6_{m,\xi_0}(u,\bar u,  u,\bar u, u,\bar u) ,
\end{equation}
\begin{equation}\label{dens-flux-p-G}
 (\partial_t  + 2 \xi_0 \partial_x)\ps_{\xi_0}(u) = \partial_x (e_{\xi_0}(u) + R^4_{p,\xi_0}(u,\bar u, u,\bar u)) +
R^6_{p,\xi_0}(u,\bar u,  u,\bar u, u,\bar u) ,
\end{equation}
where the symbols for $R^4_{m,\xi_0}$ and $R^4_{p,\xi_0}$ are defined by 
\begin{equation*}
   r^4_{m,\xi_0} =  r^4_{m} + 2 \xi_0 b^4_m, 
 \qquad   r^4_{p,\xi_0} =  r^4_{p} + 2 \xi_0 b^4_p
 + 2 \xi_0 r^4_m + 4 \xi_0^2 b^4_m.
\end{equation*}

\subsection{Localized density-flux identities for mass and momentum}

In our analysis later on, we will not use density-flux pairs for global estimates, but instead we will use them only in a frequency localized setting. 

Here we begin our discussion with a symmetric bilinear symbol
$a(\xi,\eta)$. We are assuming it generates a real valued quadratic form $A(u,\bar u)$, i.e. that 
\[
a(\xi,\eta) = \bar a(\eta,\xi),
\]
and that its symbol is bounded and uniformly smooth.
Later we will use such symbols $a$ to localize our analysis
to intervals $I$ in frequency, either of unit size or larger.

Corresponding to such $a$ we define corresponding quadratic 
localized mass, momentum and energy by
\[
m_a(\xi,\eta) = a(\xi,\eta), \qquad p_a(\xi,\eta) = -(\xi+\eta)
a(\xi,\eta), \qquad e_a(\xi,\eta) = (\xi+\eta)^2
a(\xi,\eta),
\]
A direct computation yields the relation
\begin{equation}
\partial_t M_a(u) = P_a(u) +
C_{m,a}^4(u) ,   
\end{equation}
where the symbol $C_{m,a}^4$ is given by 
\[
\begin{aligned}
c_{m,a}^4(\xi_1,\xi_2,\xi_3,\xi_4) = -\frac{i}{2} 
[ & \ c(\xi_1,\xi_2,\xi_3) m_a(\xi_1-\xi_2 +\xi_3,\xi_4) 
+c(\xi_1,\xi_4,\xi_3) m_a(\xi_1-\xi_4 +\xi_3,\xi_2)
\\ & - \bar{c}(\xi_2,\xi_3,\xi_4) m_a(\xi_3, \xi_2-\xi_3+\xi_4)
- \bar{c}(\xi_2,\xi_1,\xi_4) m_a(\xi_3, \xi_2-\xi_1+\xi_4)].
\end{aligned}
\]
A similar identity applies in the case of the localized momentum,
where we simply replace the symbol $m_a$ by $p_a$.

As before, this symbol vanishes on the resonant set $\calR$, 
so we can represent it as in the division relation \eqref{choose-R4m}, 
\begin{equation}\label{choose-R4ma}
c^4_{m,a} +  i  \Delta^4 \xi^2 b^4_{m,a} = i \Delta^4 \xi r^4_{m,a} ,
\end{equation}
as well as
\begin{equation}\label{choose-R4pa}
c^4_{p,a} +  i  \Delta^4 \xi^2 b^4_{p,a} = i \Delta^4 \xi r^4_{p,a} .
\end{equation}
Then, defining $\ms_a$ and $\ps_a$ as before, 
\begin{equation}\label{ma-sharp}
\ms_a(u) = M_a(u) + B^4_{m,a}(u,\bar u, u,\bar u),
\end{equation}
\begin{equation}\label{pa-sharp}
\ps_a(u) = P_a(u) + B^4_{p,a}(u,\bar u, u,\bar u),
\end{equation}
we obtain density-flux identities akin to
\eqref{dens-flux-m},  namely 
\begin{equation}\label{dens-flux-ma}
 \partial_t \ms_a(u) = \partial_x(P_{a}(u)
+ R^4_{m,a}(u)) +  R^6_{m,a}(u),
\end{equation}
and 
\begin{equation}\label{dens-flux-pa}
 \partial_t \ps_a(u) = \partial_x(E_{a}(u)
+ R^4_{p,a}(u)) +  R^6_{p,a}(u).
\end{equation}

We will consider these relations together with their Galilean shifts
obtaining relations of the form 
\begin{equation}\label{dens-flux-maG}
(\partial_t + 2 \xi_0 \partial_x)\ms_a(u) = \partial_x(P_{a,\xi_0}(u)
+ R^4_{m,a,\xi_0}(u)) +  R^6_{m,a,\xi_0}(u),
\end{equation}
respectively 
\begin{equation}\label{dens-flux-paG}
(\partial_t + 2 \xi_0 \partial_x)\ps_{a,\xi_0} (u) = \partial_x(E_{a,\xi_0}(u)
+ R^4_{p,a,\xi_0}(u)) +  R^6_{p,a,\xi_0}(u).
\end{equation}

These correspond to the algebraic division relations
\begin{equation}\label{choose-R4m-a}
c^4_{m,a} +  i  \Delta^4 (\xi-\xi_0)^2 b^4_{m,a} = i \Delta^4 \xi r^4_{m,a,\xi_0} ,
\end{equation}
respectively
\begin{equation}\label{choose-R4p-a}
c^4_{p,a,\xi_0} +  i  \Delta^4 (\xi-\xi_0)^2 b^4_{p,a,\xi_0} = 
i \Delta^4\xi r^4_{p,a,\xi_0} ,
\end{equation}
where
\begin{equation}\label{pa-xi0}
\begin{aligned}
c_{p,a,\xi_0}^4(\xi_1,\xi_2,\xi_3,\xi_4) = & -\frac{i}{2} 
[   c(\xi_1,\xi_2,\xi_3) p_{a,\xi_0}(\xi_1-\xi_2 +\xi_3,\xi_4) 
+c(\xi_1,\xi_4,\xi_3) p_{a,\xi_0}(\xi_1-\xi_4 +\xi_3,\xi_2)
\\ & - \bar{c}(\xi_2,\xi_3,\xi_4) p_{a,\xi_0}(\xi_1, \xi_2-\xi_3+\xi_4)
- \bar{c}(\xi_2,\xi_1,\xi_4) p_{a,\xi_0}(\xi_3, \xi_2-\xi_1+\xi_4)].
\end{aligned}
\end{equation}

The symbols above are connected in the obvious way. Precisely,
we have 
\begin{equation}
  r^4_{m,a,\xi_0} =   r^4_{m,a} + 2\xi_0 b^4_{m,a},
 \qquad 
\end{equation}
and
\begin{equation}
\ps_{a,\xi_0} = \ps_a + 2\xi_0 \ms_a, \qquad b^4_{p,a,\xi_0}
= b^4_{p,a} + 2 \xi_0 b^4_{m,a}, 
\end{equation}
and finally 
\begin{equation}
r^4_{p,a,\xi_0} = r^4_{p,a} + 2\xi_0 b^4_{p,a} + 2 \xi_0 r^4_{m,a,\xi_0}.
\end{equation}

To use these density flux relations we need to have appropriate 
bounds for our symbols:

\begin{proposition} 
\label{p:symbols}
Let $J \subset \R$ be an interval of length $r$, and $d(\xi_0,J) \lesssim r$.  Assume that $a$ is supported in $J \times J$, with bounded and uniformly smooth symbol. Then the relations \eqref{choose-R4m-a}
and \eqref{choose-R4p-a} hold with symbols $b^{4}_{m,a}$, $b^{4}_{p,a,\xi_0}$, $r^4_{m,a,\xi_0}$ and $r^4_{p,a,\xi_0}$
which can be chosen to have the following properties:

\begin{enumerate}[label=\roman*)]
    \item Support: they are all supported in the region where 
    at least one of the frequencies is in $J$.
    \item Size: 
\begin{equation}
 |b^{4}_{m,a}| \lesssim \frac{1}{\la \xih \ra \la \xim \ra }, 
 \qquad
 |b^{4}_{p,a,\xi_0}| \lesssim \frac{r}{\la \xih \ra \la \xim \ra },
\end{equation}    
\begin{equation}
\begin{aligned}
 |r^{4}_{m,a,\xi_0}| \lesssim  & \ \frac{1}{\la \xim  \ra } 1_{\Omega_1 \cup \Omega_3}
 + \frac{r}{\la \xih \ra \la \xim \ra }
   1_{\Omega_1 \cup \Omega_2},
 \\
 |R^{4}_{p,a,\xi_0}| \lesssim & \ \frac{r}{\la \xim  \ra } 1_{\Omega_1 \cup \Omega_3}
 + \frac{r^2}{\la \xih \ra \la \xim \ra }
   1_{\Omega_1 \cup \Omega_2}.
 \end{aligned}  
\end{equation}     
  \item Regularity:  similar bounds hold for all derivatives.
\end{enumerate}
\end{proposition}

\begin{proof}
 This is easily done by applying Lemma~\ref{l:division},
 see also Remark~\ref{r:division}.
\end{proof}


\section{Interaction Morawetz identities}
\label{s:Morawetz}

\subsection{The linear Schrodinger equation} 

The interaction Morawetz inequality aims to capture the fact that the 
momentum moves to the right faster than the mass. Here the left/right symmetry is broken
due to the sign choice which is implicit in the choice of the momentum.

\subsubsection{A global computation} To warm up,  we start with two solutions $u$ and $v$ for the linear Schr\"odinger equation. To these we associate
the interaction functional 
\[
\bI(u,v) = \int_{x > y} M(u)(x) P(v)(y) - P(u)(x) M(v)(y)  \,dx dy,
\]
and compute $d\bI/dt$ using the conservation laws \eqref{df-lin}. We have
\[
\begin{aligned}
\frac{d}{dt} \bI(u,v) =  & \ \int_{x > y} \partial_x P(u)(x)
P(v)(y) +  M(u)(x) \partial_y  E(v)(y)
\\ & \ \ \ \ 
-\partial_x E(u)(x)
M(v)(y) - P(u)(x) \partial_y  P(v)(y)\,
dx dy 
\\
= &  \int M(u) E(v) + M(v) E(u) 
- 2 P(u)P(v) \,dx:= 
\int J^4(u,\bu,v,\bv)\, dx
.\end{aligned}
\]
Here  $J^4$ can be chosen\footnote{ Recall that 
a-priori the symbol of $j^4$ is only determined uniquely on the diagonal $\Delta^4 \xi = 0$.}
to have  symbol
\[
j^4(\xi_1,\xi_2,\xi_3,\xi_4) =    4(\xi_1-\xi_4)(\xi_2-\xi_3).
\]
This is because of the following computation on the diagonal $\Delta^4 \xi = 0$:
\[
(\xi_1+\xi_2)^2 + (\xi_3+\xi_4)^2 - 2 (\xi_1+\xi_2)(\xi_3+\xi_4) = (\xi_1+\xi_2-\xi_3 - \xi_4)^2 
=  4 (\xi_1 - \xi_4)(\xi_2 - \xi_3).
\]
Thus we have  the positivity
\[
J^4(u,\bu,v,\bv) = 4 |\partial_x (u \bv)|^2.
\]
The above computation is classically done using integration by parts, see \cite{PV}. However, it is more interesting to do it at the symbol level because we want to apply it in a more general context. Classically this is done with $u=v$, but here 
we find it convenient to break the symmetry. Primarily, our $v$'s will be spatial translations of $u$.

\subsubsection{A frequency localized bound}
Here we start with a symbol $a$ which is localized 
on the unit scale near some frequency $\xi_0$, and
consider  the interaction Morawetz functional
\begin{equation}\label{Ia-def}
\bI_a(u,v) =  \int_{x > y} M_a(u)(x) P_a(v)(y)  -   P_a(u) (x) M_a(v)(y) \, dx dy.
\end{equation}

As above, its time derivative is 
\[
\frac{d}{dt} \bI_a(u,v) =  
\bJ^4_a(u,\bu,v,\bv),
\]
where $J_a^4$ has symbol
\[
j_a(\xi_1,\xi_2,\xi_3,\xi_4) =  4a(\xi_1,\xi_2) \ol{a(\xi_3,\xi_4)}  (\xi_1-\xi_4)(\xi_2-\xi_3).
\]
 This no longer has obvious positivity.  However, if $a$ has separated variables
\begin{equation}\label{choose-a-loc}
a(\xi,\eta) = a_0(\xi) a_0(\eta),
\end{equation}
then $J_a$ is nonnegative, 
\[
\bJ_a^4 (u,v) = 4 \int |K_a(u,\bv)|^2 \, dx,
\]
where $K_a$ has symbol $(\xi-\eta) a_0(\xi) a_0(\eta)$,
i.e.
\[
K_a(u,\bar v) = \partial_x (A_0 u \overline{A_0 v}),
\]
where $A_0$ is the multiplier associated to the symbol $a_0$. 

\subsubsection{ Interaction Morawetz for separated velocities} \label{s:AB}

Here we instead  take two symbols $a$ and $b$ localized to two frequency intervals $A$ and $B$
so that $|A|,|B| \lesssim r$ and $A$ and $B$ have separation $r$ (say $A$ is to the left of $B$). Then we 
take the interaction functional
\[
\bI_{AB} = \int_{x > y} M_A(u)(x) P_B(v)(y)    - P_A(u)(x) M_B(v)(y) \, dx dy ,
\]
or equivalently
\[
\bI_{AB} = \int_{x > y} M_A(x) P_{B,\xi_0} (y)    - P_{A,\xi_0}(x) M_B(y) \,dx dy ,
\]
where $\xi_0$ is arbitrary, but can be chosen more
efficiently at distance $O(r)$ from both $A$ and $B$.

Then we compute
\[
\frac{d}{dt} \bI_{AB} = \int M_A(u) E_B(v) + E_A(u) M_B(v)(x) - 2  P_B(u) P_A(u) \, dx :=\bJ^4_{AB}(u,\bu,v,\bv), 
\]
where $\bJ^4_{AB}$ has symbol
\[
j^4_{AB}(\xi_1,\xi_2,\xi_3,\xi_4) =  2a(\xi_1,\xi_2) b(\xi_3,\xi_4)  (\xi_1-\xi_4)(\xi_2-\xi_3).
\]
Assuming that 
\[
a(\xi,\eta) = a_0(\xi) a_0(\eta), \qquad b(\xi,\eta) = b_0(\xi) b_0(\eta),
\]
we can write $\bJ^4_{AB}$ as
\[
\bJ^4_{AB} = \int |K_{AB}(u,\bar v)|^2 \, dx,
\]
where 
\[
K_{AB}(u,\bar u) = \partial_x (A_0 u \, \overline{B_0 v}).
\]

Now the differences $(\xi_1-\xi_4)$ and $(\xi_2-\xi_3)$ have size $r$ so this leads to a bilinear $L^2$ bound
for $A_0 u \cdot \overline{B_0 u}$,
\[
\bJ_{AB} \approx r^2 \|A_0 u \cdot \overline{B_0 u}\|_{L^2}^2.
\]

\subsection{Nonlinear interaction Morawetz estimates}
Here we consider the same interaction Morawetz functional as above, but now apply it to (two) solutions for the nonlinear equation \eqref{nls}. 

\subsubsection{ A simple case }
As a starting point, here we consider density-flux pairs as in \eqref{dens-flux-m}, \eqref{dens-flux-p} to which we associate the nonlinear interaction functional 
\begin{equation}
\bI(u,v) = \iint_{x > y} \ms(u)(x) \ps(v)(y) - \ps(u)(x) \ms(v)(y) \, dx dy .    
\end{equation}
Using the density-flux relations we obtain
\begin{equation}
\frac{d\bI}{dt} = \bJ^4 + \bJ^6 + \bJ^8 + \bK^8,
\end{equation}
where $\bJ^4$ is the same as above, while $\bJ^6$ and $\bJ^8$ are given by
\begin{equation}
\begin{aligned}
\bJ^6(u,v) =  \int & \ M(u) R^4_{p}(v)+ B^4_m(u) E(v)
- P(u)B^4_p(v)- R^4_m(u)P(v) +
\\ & \ 
M(v) R^4_{p}(u)+ B^4_m(v) E(u)
- P(v)B^4_p(u)- R^4_m(v)P(u)\,  dx ,
\end{aligned}
\end{equation}
respectively 
\begin{equation}
\bJ^8(u,v) =  \int B^4_m(u) R^4_{p}(v) - R^4_{m}(u) B^4_p(v)
+ B^4_m(v) R^4_{p}(u) - R^4_{m}(v) B^4_p(u)
\,  dx .
\end{equation}
Finally, we are also left with the double integral
\begin{equation}\label{K8-def}
\begin{aligned}
\bK^8 &= \iint_{x > y} \ms(u)(x)  R^6_{p}(v)(y)  + \ps(v)(y) R^6_{m}(u)(x) \, dx dy \\ \quad &- \iint_{x > y}
\ms(v)(y)  R^6_{p}(u)(x)  + \ps(u)(x) R^6_{m}(v)(y) \,
dx dy,
\end{aligned}
\end{equation}
whose leading part has order $8$ but  also contains terms of order $10$;  but we will treat it all perturbatively later.

\bigskip

It is instructive to consider the case of the cubic defocusing NLS. There $B^4_m= 0$, $B^4_p=0$
and thus $R^6_m= 0$, $R^6_p=0$. Further, $R^4_{m} = 0$ but $R^4_p = 1$.
Thus in particular we get 
\[
\bJ^6(u,u) = \int |u|^6 \, dx.
\]
This is where the focusing/defocusing type of the equation comes in, as it determines the sign of $\bJ^6$ 
(relative to the sign of $\bJ^4$).

\subsubsection{Nonlinear interaction Morawetz: the localized diagonal case}

Here we use the frequency localized mass density-flux \eqref{dens-flux-m-G}  and the corresponding momentum density-flux \eqref{dens-flux-p-G}
in order to produce a localized interaction Morawetz estimate. We consider a smooth symbol $a$ as in \eqref{choose-a-loc}, where $a_0$ is  localized around a frequency $\xi_0$ on the unit scale.

Correspondingly, we have the localized mass and momentum densities
\[
\ms_a = M_a(u,\bar u) + B^4_{m,a}(u,\bar u, u,\bar u),
\]
\[
\ps_{a,\xi_0} = P_{a,\xi_0}(u,\bar u) + B^4_{p,a,\xi_0}(u,\bar u, u,\bar u),
\]
which satisfy the conservation laws
\[
(\partial_t + 2 \xi_0 \partial_x)\ms_a(u) = \partial_x(P_{a,\xi_0}(u)
+ R^4_{m,a,\xi_0}(u)) +  R^6_{m,a,\xi_0}(u).
\]
\[
(\partial_t + 2 \xi_0 \partial_x)\ps_{a,\xi_0}(u) = \partial_x(E_{a,\xi_0}(u)
+ R^4_{p,a,\xi_0}(u)) +  R^6_{p,a,\xi_0}(u).
\]
For these we define the interaction Morawetz functional 
\begin{equation}\label{Ia-sharp-def}
\bI_{a}(u,v) =   \iint_{x > y} \ms_a(u)(x) \ps_{a,\xi_0}(v) (y) -  
\ps_{a,\xi_0}(u)(x) \ms_{a}(v) (y) \, dx dy,
\end{equation}
where, by writing it in a symmetric fashion, we have completely
eliminated its dependence on $\xi_0$.

The time derivative of $\bI_{a}$ is
\begin{equation}\label{interaction-xi}
\frac{d}{dt} \bI_{a} =  \bJ^4_{a} + \bJ^6_{a} + \bJ^8_{a} + \bK^8_{a}, 
\end{equation}
where all the terms are  independent of $\xi_0$.

Here the quartic contribution $\bJ^4_{a}$ is 
the same as in the linear case,
\[
\bJ^4_{a}(u,v) = \int M_a(u) E_{a,\xi_0}(v) 
+ M_a(v) E_{a,\xi_0}(u)
- 2 P_{a,\xi_0}(u) P_{a,\xi_0}(v)\, dx.
\]
The sixth order term $\bJ^6_{a}$ has the form
\begin{equation}\label{J6-def}
\begin{aligned}
\bJ^6_a(u,v) =  \int & \ M_a(u) R^4_{p,a,\xi_0}(v)+ B^4_{m,a}(u) E_{a,\xi_0}(v)
- P_{a,\xi_0}(u)B^4_{p,a,\xi_0}(v)- R^4_{m,a,\xi_0}(u)P_{a,\xi_0}(v) 
\\ & \ 
+ M_a(v) R^4_{p,a,\xi_0}(u)+ B^4_{m,a}(v) E_{a,\xi_0}(u)
- P_{a,\xi_0}(v)B^4_{p,a,\xi_0}(u)- R^4_{m,a,\xi_0}(v)P_{a,\xi_0}(u)\,  dx .
\end{aligned}
\end{equation}
Next we have 
\begin{equation}\label{J8-def}
\bJ^8_{a}(u,v) =   \int 
B^4_{m,a}(u) R^4_{p,a,\xi_0}(v) - R^4_{m,a,\xi_0}(u) B^4_{p,a,\xi_0}(v)
+ B^4_{m,a}(v) R^4_{p,a,\xi_0}(u) - R^4_{m,a,\xi_0}(v) B^4_{p,a,\xi_0}(u)
 \, dx .
\end{equation}

Finally the $8$-linear term $\bK^8_{a}$ has the form
\begin{equation}\label{K8-def-ab}
\begin{aligned}
\bK^8_a(u,v) = \iint_{x > y} & \ \ms_a(u)(x)  R^6_{p,a,\xi_0}(v)(y)  + \ps_{a,\xi_0}(v)(y) R^6_{m,a}(u)(x) 
\\ & \ - 
\ms_a(v)(y)  R^6_{p,a,\xi_0}(u)(x)  - \ps_{a,\xi_0}(u)(x) R^6_{m,a}(v)(y) \,
dx dy.
\end{aligned}
\end{equation}
This also includes  a $10$-linear term.

Importantly, here we compute the symbol of $\bJ^6_{a,\xi_0}$ on the diagonal
$\xi_1 = \xi_2=\xi_3=\xi_4=\xi_5=\xi_6:=\xi$. This will be essential later on in order to obtain bounds for the $L^6$ Strichartz norm.

\begin{lemma}
The diagonal trace of the symbol $j^6_{a}$
is 
\begin{equation}\label{good-J6}
j^6_{a}(\xi) = a^2 (\xi) c (\xi, \xi, \xi) .
\end{equation}
\end{lemma}
\begin{proof}
Since our symbol does not actually depend on $\xi_0$, it suffices to compute it 
at $\xi = \xi_0$. The advantage is that $p_{a,\xi_0}(\xi_0) = e_{a,\xi_0}(\xi_0)=0$,
so we are left with the simpler expression
\[
j^6_{a}(\xi_0) = m_a(\xi_0) r^4_{p,a,\xi_0}(\xi_0).
\]
For $r^4_{p,a,\xi_0}$ we have the relation 
\[
c^4_{p,a,\xi_0} +  i  \Delta^4 (\xi-\xi_0)^2 b^4_{p,a,\xi_0} = i \Delta^4 \xi r^4_{p,a,\xi_0} .
\]
We differentiate with respect to $\xi_1$ and then set 
all $\xi$'s equal to obtain
\[
\partial_1 c^4_{p,a,\xi_0} (\xi_0) = i r^4_{p,a,\xi_0} (\xi_0).
\]
It remains to compute the $\xi_1$ derivative of 
$c^4_{p,a,\xi_0}$ on the diagonal. This is a direct 
computation using the formula \eqref{pa-xi0}. 
Recalling that 
\[
p_a(\xi_1,\xi_2) = m_a (\xi_1,\xi_2)(-\xi_1-\xi_2+2\xi_0),
\]
it follows that 
\[
r^4_{p,a,\xi_0} (\xi_0) = m_a(\xi_0) c(\xi_0),
\]
where we recall that $c$ is real on the diagonal.
Therefore
\begin{equation*}
j^6_{a,\xi_0}(\xi_0) = a^2 (\xi_0) c(\xi_0) ,
\end{equation*}
as needed.
\end{proof}

\subsubsection{Nonlinear interaction Morawetz: the transversal case}

Here we return to the setting of Section~\ref{s:AB}
where we have two frequency intervals $A,B$ with size 
at most $M$ and separation also $M$, and two smooth and bounded symbols $a,b$ which are localized in the two intervals.
Our interaction Morawetz functional is given by
\begin{equation}\label{interaction-bi}
 \bI_{AB} = \int_{x > y} \ms_a(u)(x) \ps_{b,\xi_0}(v)(y)    - \ps_{a,\xi_0}(u)(x) \ms_b(v)(y)\, dx dy ,
\end{equation}
and we observe as before that this does not depend on $\xi_0$.

Using again the frequency localized mass density-flux \eqref{dens-flux-m-G}  and the corresponding momentum density-flux \eqref{dens-flux-p-G}
we produce a localized interaction Morawetz estimate,
\begin{equation}\label{interaction-xi-AB}
\frac{d}{dt} \bI_{AB} =  \bJ^4_{AB} + \bJ^6_{AB} + \bJ^8_{AB} + \bK^8_{AB} .
\end{equation}

Here the quartic contribution $\bJ^4_{AB}$ is 
the same as in the linear case,
\[
\bJ^4_{AB} = \int M_a(u)(x) E_{b,\xi_0}(v)(x) + M_b(v)(x) E_{a,\xi_0}(u)(x)-  2 P_{a,\xi_0}(u)(x) P_{b,\xi_0}(v)(x)\, dx,
\]
and captures the 
bilinear $L^2$ bound.

The sixth order term $\bJ^6_{AB}$ has the form
\[
\bJ^6_{AB} =   \int -( P_{a,\xi_0} B^4_{p,b,\xi_0}
+ P_{b,\xi_0} R^4_{m,a,\xi_0}) + ( 
M_a R^4_{p,b,\xi_0} +E_{b,\xi_0} B^4_{m,a,\xi_0})      - \text{symmetric} \, dx,
\]
where in the symmetric part we interchange both the indices $a,b$ and the functions $u,v$.

Next we have 
\[
\bJ^8_{AB} =   \int - R^4_{m,a,\xi_0} B^4_{p,b,\xi_0}
 +  
 B^4_{m,a,\xi_0} R^4_{p,b,\xi_0}  - \text{symmetric} \,      dx .
\]

Finally the $8$-linear term $\bK^8_{a,\xi_0}$ has the form
\[
\bK^8_{AB} = \iint_{x > y} \ms_a(x)  R^6_{p,b,\xi_0} + \ps_{b,\xi_0} R^6_{m,a,\xi_0} - \text{symmetric} \, dx dy.
\]
 As before this also includes  a $10$-linear term.


\section{Frequency envelopes and the bootstrap argument}

\label{s:boot}

The primary goal of the proof of our main result in Theorem~\ref{t:main} is to establish a global $L^\infty_t L^2_x$ bound for small data solutions; by the local well-posedness result in Section~\ref{s:local}, this implies
the desired global well-posedness result. However, along the way, we will also establish bilinear $L^2$ and Strichartz bounds for the solutions. These will both play an essential role in the proof of Theorem~\ref{t:main},
and will  also establish the scattering properties of our global solutions.

Since the proof of our estimates loops back in a complex manner, it is most convenient to establish the bilinear $L^2$ and the $L^6$ Strichartz bounds in the setting of a bootstrap argument, where we already assume that the desired bilinear and Strichartz estimates hold but with weaker constants. 

The set-up for the bootstrap is most conveniently described using the language  of frequency envelopes. This was originally introduced  in work of Tao, see e.g. \cite{Tao-WM}, but in the context of dyadic Littlewood-Paley decompositions.  But here instead we work with 
a uniform decomposition on the unit scale, which requires a substantial revision of the notion of ``slowly varying", which we replace by 
the new notion  of ``maximal property" introduced in Section~\ref{s:fe}.

\bigskip

To start with, we assume that the initial data has small size,
\[
\| \du_0\|_{L^2} \lesssim \epsilon.
\]
We consider a frequency decomposition for the initial  data on a unit spatial scale, 
\[
\du_0 = \sum_{k \in \Z} \du_{0,k}.
\]
Then we place the initial data components under an admissible  frequency envelope on the unit scale,
\[
\|\du_{0,k}\|_{L^2} \leq \epsilon c_k, \qquad c \in \ell^2,
\]
where the envelope $\{c_k\}$ is not too large,
\begin{equation*}
\| c\|_{\ell^2} \approx 1.    
\end{equation*}

Our goal will be to establish the following frequency envelope bounds for the solution:

\begin{theorem}\label{t:boot}
 Let $u \in C([0,T];L^2)$ be a solution for the equation \eqref{nls}
with initial data $\du_0$  which has $L^2$ size at most $\epsilon$.
Let $\{\epsilon c_k\}$ be an admissible frequency envelope for the initial data 
in $L^2$, with $c_k$ normalized in $\ell^2$. Then the solution $u$ satisfies 
the following bounds:
\begin{enumerate}[label=(\roman*)]
\item Uniform frequency envelope bound:
\begin{equation}\label{uk-ee}
\| u_k \|_{L^\infty_t L^2_x} \lesssim \epsilon c_k,
\end{equation}
\item Localized Strichartz bound:
\begin{equation}\label{uk-se}
\| u_k \|_{L_{t,x}^6} \lesssim (\epsilon c_k)^\frac23,
\end{equation}
\item Localized Interaction Morawetz:
\begin{equation}\label{uk-bi}
\| \partial_x |u_k|^2  \|_{L^2_{t,x}} \lesssim \epsilon^2 c_k^2,
\end{equation}
\item Transversal bilinear $L^2$ bound:
\begin{equation} \label{uab-bi}
\| \partial_x(u_{A} \bu_B(\cdot+x_0))  \|_{L^2_{t,x}} \lesssim \epsilon^2 c_A c_B\, \la dist(A,B)\ra^{\frac12},
\end{equation}
 for all $x_0 \in \R$ whenever $|A| + |B| \lesssim  \la dist(A,B)\ra $.
\end{enumerate}
\end{theorem}
Here \eqref{uk-bi} can be seen as a particular case of \eqref{uab-bi}
when $A=B$ have unit length; we stated it separately in order 
to ease comparison with earlier work on Interaction Morawetz estimates.

To prove this theorem, we make a bootstrap assumption where we assume the same bounds but with a worse constant $C$, as follows:

\begin{enumerate}[label=(\roman*)]
\item Uniform frequency envelope bound,
\begin{equation}\label{uk-ee-boot}
\| u_k \|_{L^\infty_t L^2_x} \lesssim C \epsilon c_k,
\end{equation}
\item Localized Strichartz bound,
\begin{equation}\label{uk-se-boot}
\| u_k \|_{L^6_{t,x}} \lesssim C(\epsilon c_k)^\frac23,
\end{equation}
\item Localized Interaction Morawetz,
\begin{equation}\label{uk-bi-boot}
\| \partial_x |u_k|^2  \|_{L^2_{t,x}} \lesssim  C\epsilon^2 c_k^2,
\end{equation}
\item Transversal Interaction Morawetz, 
\begin{equation} \label{uab-bi-boot}
\| \partial_x (u_{k_1} \bar u_{k_2}(\cdot+x_0))  \|_{L^2_{t,x}} \lesssim C \epsilon^2 c_{k_1} c_{k_2} \la k_1-k_2\ra^{\frac12}
\end{equation}
uniformly for all $x_0 \in \R$.
\end{enumerate}

Then we seek to improve the constant in these bounds. The gain will come from the fact that the $C$'s will always come paired
with extra $\epsilon$'s. 

We remark that the bootstrap hypothesis
for the transversal bilinear $L^2$ bound \eqref{uab-bi-boot} 
only requires unit size localization, unlike the corresponding conclusion
\eqref{uab-bi}.  On one hand this simplifies the continuity argument closing the bootstrap. On another hand this is related to the fact that closing the bootstrap argument for global well-posedness only requires \eqref{uab-bi} for unit size sets. The full bound \eqref{uab-bi}
is only used in the last section in order to obtain the global 
Strichartz and bilinear $L^2$ bounds, which are of course very interesting but secondary to the proof of the global result.

We also remark on the need to add translations to the bilinear 
$L^2$ estimates. This is because, unlike the linear bounds
\eqref{uk-ee-boot} and \eqref{uk-se-boot} which are inherently invariant with respect to translations, bilinear estimates are not invariant 
with respect to separate translations for the two factors.
One immediate corollary of \eqref{uab-bi-boot} is that for any multipliers $L_1$ and $L_2$ with smooth and bounded symbols we have 
\begin{equation} \label{uab-bi-boot-trans}
\| \partial_x (L_1(D) u_{k_1} \ol{L_2(D) u_{k_2}}(\cdot+x_0))  \|_{L^2} \lesssim C \epsilon^2 c_{k_1} c_{k_2} \la k_1-k_2\ra^{\frac12}.
\end{equation}
This is essentially the only way we will use this translation invariance
in our proofs.

For the rest of this section, we provide the continuity argument which shows that it suffices  to prove Theorem~\ref{t:boot} under the bootstrap assumptions \eqref{uk-ee-boot}-\eqref{uab-bi-boot}. 

For this, we denote by $T$ the maximal time for which the bounds 
\eqref{uk-ee-boot}-\eqref{uab-bi-boot} hold in $[0,T]$. By the local well-posedness result we have $T \geq 1$. Assume by contradiction that $T$
is finite.
Then the bootstrap version of the theorem implies that the bounds \eqref{uk-ee}-\eqref{uab-bi}
hold in $[0,T]$. In particular, $u(T)$ will also be controlled 
by the same maximal envelope $c_j$ coming from the initial data.
By the local well-posedness result, this implies in turn that 
the bounds  \eqref{uk-ee-boot}-\eqref{uab-bi-boot} hold in $[T,T+1]$
with $C \approx 1$. Adding this to the bounds \eqref{uk-ee}-\eqref{uab-bi}
in $[0,T]$, it follows that  \eqref{uk-ee-boot}-\eqref{uab-bi-boot} hold in $[0,T+1]$, thereby contradicting the maximality of $T$.

  
\section{The frequency envelope bounds}
\label{s:fe-bounds}

The aim of this section is to prove the frequency envelope bounds in 
Theorem~\ref{t:boot}, given the bootstrap assumptions \eqref{uk-ee-boot}-\eqref{uab-bi-boot}. In the proof we will rely
on our modified energy and momentum functionals, whose components
we estimate first. The frequency localized energy estimate \eqref{uk-ee} will be
an immediate consequence of these bounds. For the Strichartz and $L^2$
bilinear bounds we will then use the interaction Morawetz identities,
first in a localized diagonal setting and then in a transversal setting.

\subsection{Spatial and space-time \texorpdfstring{$L^1$}{} bounds}
Here we consider the corrections $B^4_{m,a}$ and errors $R^6_{m,a}$ 
and their momentum counterparts associated to a smooth bump function $a$ selecting a frequency interval $A \subset \Z$. For  $B^4_{m,a}$ we will prove a fixed time $L^1$ bound, while for $R^6_{m,a}$ we will prove a space-time $L^1$ bound.
These bounds will be repeatedly used in each of the following subsections, first in the case when $A$ has unit size and then
in the case when $A$ has a larger size. We begin with the 
$B^4_{m,a}$ bound.

\begin{lemma}\label{l:B4-multi}
Assume that the bootstrap bound \eqref{uk-ee-boot} holds. Then we have the fixed time estimate 
\begin{equation}\label{b4-ma}
\| B^4_{m,a}(u) \|_{L^1_x} \lesssim \epsilon^4 C^4
c_A^2.
\end{equation}
\end{lemma}
The corresponding bound for the momentum follows as a corollary, once 
we add an additional assumption in order to fix the momentum size:

\begin{corollary}\label{c:B4-multi}
Assume that the bootstrap bound \eqref{uk-ee-boot} holds. 
Let $\xi_0 \in \R$, and
\[
n = \max_{k \in A} |k -\xi_0|.
\]
Then we have the fixed time estimate 
\begin{equation}\label{b4-pa}
 \| B^4_{p,a,\xi_0}(u) \|_{L^1_x} \lesssim n \epsilon^4 C^4 c_A^2,
\end{equation}
\end{corollary}

\begin{proof}
The bounds \eqref{b4-ma} and \eqref{b4-pa} are similar, the only difference arises from the additional $n$ factor in the size of the symbol $p_A$. So we will prove the first bound.
Using our partition of unity in frequency on the unit scale we 
expand 
\[
 B^4_{m,a}(u) = \sum_{k_1,k_2,k_3,k_4 \in \Z} B^4_{m,A}(u_{k_1},\bu_{k_2},u_{k_3},\bu_{k_4}).
 \]
Here we will separately estimate each term in $L^1_x$ based on the size of the symbol. By Proposition~\ref{p:symbols}, for frequencies within a unit neighbourhood of $[k]=(k_1,k_2,k_3,k_4)$ the symbol $b^4_{m,a}$
and its derivatives can be estimated by 
\[
b^4_{m,a}[k]:= \frac{1}{\la \delta k^{hi}\ra \la \delta k^{med}\ra}.
\]
In addition, its support is contained in the region $\Omega_1 \cup \Omega_2$, where at least one frequency is in $A$. The region $\Omega_1 \cup \Omega_2$ can be described as the set of those quadruples $[k]$
so that 
\begin{equation}\label{which-k}
\text{ either} \quad |\Delta^4 k| \lesssim 1, \quad \text{or} \quad
 |\Delta^4 k| \ll \delta k^{med}.
\end{equation}

Without loss in generality we assume that $k_1 \in A$. 
Then, using the above properties, we can estimate the $L^1_x$ bound in the lemma as 
\[
\| B^4_{m,a}(u) \|_{L^1_x} \lesssim \epsilon^4 \sum_{k_1 \in A} \sum_{[k] \in \Omega_1 \cup \Omega_2} \frac{1}{\la \delta k^{hi}\ra \la \delta k^{med}\ra}
c_{k_1} c_{k_2} c_{k_3} c_{k_4}.
\]
Fixing $k_1 \in A$, it suffices to show that 
\begin{equation}\label{fix-k1}
S_{k_1} :=  \sum_{k_2,k_3,k_4} \frac{1}{\la \delta k^{hi}\ra \la \delta k^{med}\ra} c_{k_2} c_{k_3} c_{k_4} \lesssim c_{k_1}.
\end{equation}
This no longer has anything to do with the set $A$. For later 
use we have also removed the  restriction $[k] \in \Omega_1 \cup \Omega_2$.

To discuss the possible configurations for $[k]$
we denote by $n_1 \leq n_2$ the dyadic size of $\delta k^{med}$,
respectively $\delta k^{hi}$.  By Galilean invariance we set 
$k_1=0$, and then the rest of the indices may be reordered so that  
\[
 |k_2| \lesssim  n_1, \qquad |k_3| \lesssim n_2, \quad  |k_4| \approx n_2, \qquad |k_3-k_4| \lesssim n_1.
\]
Then we have
\[
S \lesssim \sum_{n_1 \leq n_2} \sum_{|k_2| \lesssim n_1}
\sum_{|k_4| \approx n_2}
\sum_{|k_3-k_4| \lesssim n_1}
\frac{1}{n_1 n_2} c_{k_2} c_{k_3} c_{k_4} .
\]
Here we use twice the envelope maximal function bound  to write
\[
\frac{1}{n_1} \sum_{|k_2| \lesssim n_1} c_{k_2} \lesssim c_0,
\qquad 
\frac{1}{n_1} \sum_{|k_3-k_4| \lesssim n_1} c_{k_3} \lesssim c_{k_4}.
\]
This gives
\[
S \lesssim  c_0 
\sum_{n_1 \leq n_2} \frac{n_1}{n_2} \sum_{|k_4| \approx n_2}  c_{k_4}^2 
\approx c_0 \sum_{n_2} \sum_{|k_4| \approx n_2} c_{k_4}^2\lesssim c_0.
\]
This concludes the proof of \eqref{fix-k1}, and therefore the proof of the lemma.
\end{proof}

Next we turn our attention to $R^6_{m,a}$, which we estimate as follows:

\begin{lemma}\label{l:R6-AB}
Under our bootstrap assumptions \eqref{uk-ee-boot}-\eqref{uab-bi-boot}, we have the space-time bound
\begin{equation}\label{R6-m-bd}
\|R^6_{m,a}\|_{L^1_{t,x}} \lesssim \epsilon^4 C^6 c_A^2.
\end{equation}
\end{lemma}

As above, we also have a similar bound for the momentum:

\begin{corollary}\label{c:R6-AB}
Assume that the bootstrap bounds \eqref{uk-ee-boot}-\eqref{uab-bi-boot}
hold. Let $\xi_0 \in \R$, and
\[
n = \max_{k \in A} |k -\xi_0|.
\]
Then we have  the space-time bound
\begin{equation}\label{R6-p-bd}
\|R^6_{p,a,\xi_0}\|_{L^1_{t,x}}
\lesssim n \epsilon^4 C^6 c_A^2.  
\end{equation}
\end{corollary}

\begin{proof} 
As in the case of the earlier fixed time bound, we will focus 
on \eqref{R6-m-bd}, as the proof of \eqref{R6-p-bd} is essentially the same.
We recall $R^6_{m,a}$ is obtained from the cubic terms in the time derivative of $B^4_{m,a}$. We denote the four frequencies in $B^4_{m,a}$ by $k_0,k_1,k_2,k_3$,
where the $k_0$ factor gets differentiated in time. 
One of these four frequencies, call it $k_A$, must be in $A$.

With $k_A$ as above, we expand 
\[
R^6_{m,a}(u) = \sum_{k_A \in A} R^6_{m,a,k_A}(u).
\]
Then it suffices to establish the bound
\begin{equation}
\|   R^6_{m,a,k_A}(u) \|_{L^1_{t,x}} \lesssim \epsilon^4 C^6 c_{k_A}^2.  \end{equation}
Here by Galilean invariance we can set $k_A=0$. We also drop the index $A$, as  no localization associated to the set $A$ will  be used in the sequel. In particular we replace $B^4_{m,a}$ by $B^4_{m,0}$ to emphasize that one of the frequencies in $B^4$
is assumed to be near zero.

 To describe  the size and localization of the symbol $b^4_{m,0}$
we introduce as before the notations $\delta k^{med}$ and  $\delta k^{hi}$ for the distances between $k_0,k_1,k_2,k_3$,
$n_1 < n_2$ for the dyadic size of  $\delta k^{med}$ and  $\delta k^{hi}$, and $\Delta^4 k$ associated to the same indices. In the support of $b^4_{m,0}$ we must have 
\begin{equation}\label{d4-k}
|\Delta^4 k| \lesssim 1 \qquad \text{or} \qquad  |\Delta^4 k| \ll n_1, 
\qquad 0 \in \{k_0,k_1,k_2,k_3\}.
\end{equation}
In this region, the symbol of $b^4_{m,0}$ as well as its derivatives have size
\begin{equation}
|b^4_{m,0}| \lesssim \frac{1}{n_1 n_2}. 
\end{equation}
The time differentiation is producing three additional frequencies $k_4,k_5, k_6$, so that 
\begin{equation}\label{gamma2}
k_0 = k_4-k_5+k_6    .
\end{equation}
Then \eqref{d4-k} translates to 
\begin{equation}\label{d6-k}
|\Delta^6 k| \lesssim 1 \qquad \text{or} \qquad |\Delta^6 k| \ll n_1
\end{equation}
relative to the indices $k_1,\cdots, k_6$.

Overall, for $R^{6}_{m,0}$ we have the decomposition
\[
R^{6}_{m,0}(u) = \sum_{n_1 \leq n_2} \sum_{k_{0-7} \in \Gamma}
R^{6}_{m,0}(u_{k_1},\bu_{k_2}, u_{k_3},\bu_{k_4},u_{k_5},\bu_{k_6}),
\]
where $\Gamma$ describes the set of indices 
satisfying \eqref{d4-k} and \eqref{gamma2}. To bound this sum in $L^1_{t,x}$ we consider several cases:
\bigskip

A. If all six frequencies are near $0$, then we use the localized $L^6$ bound to 
obtain 
\[
\|R^{6}_{m,0}(u_{k_1},\bu_{k_2}, u_{k_3},\bu_{k_4},u_{k_5},\bu_{k_6}) \|_{L^1_{t,x}}
\lesssim C^4 (\epsilon c_0)^{4},
\]
which suffices.

\bigskip

B. Otherwise,  we denote by $1 \ll n$ the minimal dyadic size of the interval  containing all six $k$ indices. Clearly 
we have $n_1 \leq n_2 \leq n$.
Also due to 
\eqref{d6-k} we must also have $|\Delta^6 k| \ll n$.
This implies that within the set $(k_1,\cdots,k_6)$
there must be at least two disjoint pairs of frequencies at distance  comparable to $n$. Applying two bilinear $L^2$
estimates, and $L^\infty$ bounds for the other two factors, 
we can bound 
\begin{equation}
   \|R^{6}_{m,0}(u)\|_{L^1_{t,x}} \lesssim \epsilon^6 C^4 S, 
   \qquad S = 
   \sum_{n_1 \leq n_2} \sum_{k_{0-6} \in \Gamma} 
   \frac{1}{n_1 n_2 n} c_{k_1} c_{k_2}c_{k_3} c_{k_4} c_{k_5} c_{k_6}.
\end{equation}
It remains to bound the above sum $S$ by 
\begin{equation}
S \lesssim c_0^2.    
\end{equation}
There are several cases to consider:

\bigskip

B1. $k_0 = 0$. Relabeling, we may assume that
\begin{equation}\label{gamma1a}
|k_1| \approx n_1, \quad |k_2-k_3| \approx n_1, \quad
|k_2|, |k_3| \approx n_2.
\end{equation}
We distinguish further cases by comparing $n_2$ and $n$.

\medskip

B1a. $n \gg n_2$.
Then we may assume that 
\begin{equation}\label{big-k}
|k_4| \lesssim |k_5| \approx |k_6| \approx n,
\end{equation}
For fixed $k_4$ we can apply the Cauchy-Schwarz inequality for the 
pair of indices $(k_5,k_6)$, and also for $(k_2,k_3)$.
We obtain
\[
S \lesssim  \sum_{n_1 \leq n_2 \ll n} 
\sum_{|k_1| \approx n_1, |k_4| \lesssim n} \frac{1}{n_2 n}
c_{k_1} c_{k_4} c_{n_2}^2 c_{n}^2
=  \sum_{n_2 \ll n} 
\sum_{|k_1| \leq n_2, |k_4| \lesssim n} \frac{1}{n_2 n}
c_{k_1} c_{k_4} c_{n_2}^2 c_{n}^2.
\]
Now we use twice the envelope maximal bound for 
the $k_1$, respectively $k_4$ summation to get
\[
S \lesssim c_0^2 \sum_{n_2 \ll n} c_{n_2}^2 c_{n}^2 \approx c_0^2.
\]

B1b. $k_0 = 0$, $n \approx n_2$.  In this case we can introduce another dyadic parameter $n_3 \leq n$ so that, after relabeling,
\begin{equation}\label{m3}
|k_4| \leq |k_5| \approx |k_6| \approx n_3.
\end{equation}
Then applying Cauchy-Schwarz inequality exactly as above 
we arrive at 
\[
S \lesssim \sum_{n_3 \leq n_2} 
\sum_{|k_1| \leq n_2, |k_4| < n_3} \frac{1}{n_2^2}
c_{k_1} c_{k_4} c_{n_2}^2 c_{n_3}^2,
\]
where we can conclude again by applying twice the envelope maximal bound for the $k_1$, respectively the $k_4$ summation relative to $0$.

\bigskip

B2. $k_1 = 0$, $|k_0| \approx n_1$. In this case we must also have 
\begin{equation}\label{gamma1b}
|k_2-k_3| \approx n_1, \quad
|k_2|, |k_3| \approx n_2.
\end{equation}
Again we compare $n$ and $n_2$:

\medskip
 
B2a. $n \gg n_2$. Here we can assume again that \eqref{big-k} holds. As in case B1a we apply 
Cauchy-Schwarz inequality for the 
pair of indices $(k_5,k_6)$, and also for $(k_3,k_2)$,
with the difference that now the difference $k_5-k_6$
is no longer fixed, instead it varies in an $n_1$ range.
Thus we lose two $n_1$ factors, obtaining 
\[
S \lesssim  c_{0} \sum_{n_1 \leq n_2 \ll n} 
\sum_{|k_4| < n} \frac{n_1}{n_2 n} c_{k_4} c_{n_2}^2 c_{n}^2.
\]
The $n_1$ summation is trivial now, and for the $k_4$
summation we use the envelope maximal bound to obtain
\[
S \lesssim  c_{0}^2 \sum_{n_2 \ll n} 
 c_{n_2}^2 c_{n}^2 \approx c_0^2.
\]

\medskip

B2b.  $n \approx  n_2$. Here we take two subcases. 

\medskip

B2b(i). If 
\[
|k_4|+|k_5|+|k_6| \lesssim n_1,
\]
then we use Cauchy-Schwarz  inequality for the pair 
$(k_2,k_3)$ to obtain
\[
S \lesssim c_0 \sum_{n_1 \leq n_2} \frac{1}{n_2^2} c_{n_2}^2 \sum_{|k_4|+|k_5|+|k_6| \lesssim n_1}
c_{k_4} c_{k_5} c_{k_6} .
\]
Finally we use the envelope maximal bound for the $k_4$
summation relative to $0$
and for $k_5$ relative to $k_6$ to get 
\[
S \lesssim c_0^2 \sum_{n_1 \leq n_2} \frac{n_1^2}{n_2^2} c_{n_2}^2  c_{\leq n_1}^2 
\lesssim 
 c_0^2 \sum_{ n_2}  c_{n_2}^2  c_{\leq n_2}^2 
\lesssim c_0^2,
\]
which suffices.

\medskip

B2b(ii). If instead 
\[
|k_4|+|k_5|+|k_6| \gg n_1,
\]
then we can introduce $n_3$ as in \eqref{m3}, with 
$ n_1 \ll n_3 \leq n_2$. Applying Cauchy-Schwarz inequality for the 
pair of indices $(k_5,k_6)$, and  $(k_2,k_3)$ yields
\[
S \lesssim c_0  \sum_{n_1 \ll n_3 \leq n_2} 
\sum_{|k_4| < n_3} \frac{n_1}{n_2^2}
 c_{k_4} c_{n_2}^2 c_{n_3}^2 
 \approx c_0  \sum_{n_3 \leq n_2} 
\sum_{|k_4| < n_3} \frac{n_3}{n_2^2}
 c_{k_4} c_{n_2}^2 c_{n_3}^2 .
\]
At this stage we complete the argument by  
using the envelope maximal bound for the $k_4$
summation.

\bigskip

B3. $k_1 = 0$, $|k_0| \approx n_2 \gg n_1$. In this case we may assume that
\begin{equation}\label{gamma1c}
|k_2| \approx n_1, \quad |k_3| \approx n_2, \qquad |k_0-k_3| \approx n_1.
\end{equation}
Next we compare $n_2$ and $n$:

\medskip

B3a. $n_2 \ll n$. Retaining $k_0$ as a summation index, we first 
apply Cauchy-Schwarz inequality for the pair $(k_5,k_6)$ to obtain 
\[
\begin{aligned}
S \lesssim & \ c_0 \sum_{n_1 \ll n_2  \ll n} \sum_{|k_0| \approx n_2} \sum_{|k_2| \approx n_1}
\sum_{|k_0 - k_3| \approx n_1} \sum_{|k_4| \lesssim n} 
\frac{1}{n_1 n_2 n}c_{k_2} c_{k_3} c_{k_4} c_n^2
\\
\lesssim & \ c_0 \sum_{n_1 \ll n_2  \ll n} \sum_{|k_3| \approx n_2} \sum_{|k_2| \approx n_1}
 \sum_{|k_4| \lesssim n} 
\frac{1}{ n_2 n}c_{k_2} c_{k_3} c_{k_4}c_n^2
\\
= & \ c_0 \sum_{n_2  \ll n} \sum_{|k_3| \approx n_2} \sum_{|k_2| \ll n_2}
 \sum_{|k_4| \lesssim n} 
\frac{1}{ n_2 n} c_{k_2} c_{k_3} c_{k_4}c_n^2.
\end{aligned}
\]
Now we use  the envelope maximal bound for  $k_4$ relative to $0$
and for $k_2$ relative to $k_3$.
This yields
\[
S \lesssim c_0^2 \sum_{n_2  \ll n}
\sum_{|k_3| \approx n_2} 
c_{k_3}^2 c_n^2
\approx c_0^2.
\]
\bigskip

B3b. $n = n_2$. In this case we dispense with $k_0$ as a summation index,
retaining instead the relation 
\[
| k_3 - k_4 + k_5-k_6 | \lesssim n_1.
\]
At least one of the frequencies $k_4,k_5,k_6$ must have size $n_2$,
say $|k_6| \approx n_2$. Then we use Cauchy-Schwarz inequality for the 
$(k_3,k_6)$ pair, losing an $n_1$ factor due to the relation above, and arriving at 
\[
S \lesssim c_0 \sum_{n_1 \leq n_2 }  \sum_{|k_2| \approx n_1}
\sum_{|k_4|,|k_5| \leq n_2}  
\frac{1}{n_2^2}c_{k_2} c_{k_4} c_{k_5} c_{n_2}^2
\approx 
c_0 
\sum_{|k_2|,|k_4|,|k_5| \leq n_2}  
\frac{1}{n_2^2}c_{k_2} c_{k_4} c_{k_5} c_{n_2}^2.
\]
Finally we use the envelope maximal bound for $k_2$ relative to $0$
and for $k_4$ relative to $k_5$
to obtain 
\[
S \lesssim  c_0^2   
c_{\leq n_2}^2  c_{n_2}^2 \lesssim c_0^2.
\]
This concludes the proof of the lemma.

\bigskip

B4. $k_2=0$. Here we can assume that 
\[
|k_1| \approx n_1, \qquad |k_3| \approx |n_2|,
\]
but the size of $k_0$ is both not set and not needed.
Instead, we will simply rely on \eqref{d6-k} and consider two subcases.

\bigskip

B4a. $n_2 \ll n$, where we can assume that \eqref{big-k} holds. Here we first use the maximal function
for $c_{k_1}$ to estimate
\[
S \leq c_0^2 \sum_{n_1 \leq  n_2 \ll n} \frac{1}{n_2n}  
c_{k_3} c_{k_4} c_{k_5} c_{k_6},
\]
where we retain the constraint relative to $k_3,k_4,k_5,k_6$,
\[
|\Delta^4 k | \lesssim k_1.
\]
Here we can fix $\Delta^4 k$ at the expense of another $n_1$ factor.
Then fixing $k_3$ and $k_4$ fixes the difference $k_5-k_6$, so applying Cauchy-Schwarz inequality with respect to $k_5$, $k_6$ we arrive at
\[
S \leq c_0^2 \sum_{n_1 \leq  n_2 \ll n} \frac{n_1}{n_2n}  
c_{k_3} c_{k_4} c_{n}^2.
\]
Finally, using H\"older's inequality for $k_3$ and $k_4$, which have size $n_2$, respectively $\leq h$ yields
\[
S \leq c_0^2 \sum_{n_1 \leq  n_2 \ll n} \frac{n_1}{n_2n}  
\sqrt{n_2 n} c_{n_2} c_{\leq n} c_{n}^2 \lesssim 
c_0^2 \sum_n c^2_{\leq n} c_{n}^2  \lesssim c_0^2.
\]

\bigskip

B4b. $n_2 \approx n$. Here the case $n_1 \approx n_2$
is straightforward, as we can directly apply once the Cauchy-Schwartz inequality for two size $n$ frequencies,
twice  H\"older's inequality and once the maximal function bound for the three remaining frequencies of size $\lesssim n$. We are left with the more interesting  case
when $n_1 \ll n$.
There, using again the  maximal function
for $c_{k_1}$ we estimate
\[
S \leq c_0^2 \sum_{n_1 \leq  n_2 \ll n} \frac{1}{n^2}  
c_{k_3} c_{k_4} c_{k_5} c_{k_6},
\]
where for the four remaining indices we have 
$|\Delta^4 k| \lesssim n_1 \ll n$. Here $|k_3| \approx n$, so there must be at least one other frequency of size $n$. Then, as in the previous case, we apply once Cauchy-Schwartz inequality for the two size $n$ frequencies, and twice  H\"older's inequality  for the two remaining frequencies of size $\lesssim n$. This concludes the proof of the lemma.

\end{proof}

\subsection{The energy estimate}
Our objective here is to prove the bound \eqref{uk-ee}. We remark that once 
this is proved, we may drop the $C^4$ factor in Lemma~\ref{l:B4-multi}.
By the Galilean invariance, it suffices to prove the desired bound \eqref{uk-ee} at $k = 0$. For this we consider a symbol $a(\xi_1,\xi_2)$ of the form 
\begin{equation}\label{choose-a}
a(\xi_1,\xi_2) = a_0(\xi_1) a_0(\xi_2),
\end{equation}
with $a_0$ localized near frequency $0$ on the unit scale.
Then
\[
\bM_a(u) = \| A_0(D) u\|_{L^2}^2,
\]
and we need to bound this quantity uniformly in time,
\begin{equation}\label{unif-0}
\bM_a(u) \lesssim c_0^2 \epsilon^2 .
\end{equation}
For this we use the density-flux relation \eqref{dens-flux-maG}
with $\xi_0 = 0$, which yields
\[
\frac{d}{dt} \ms_a(u) = \partial_x (P_a(u) + R^4_{m,a}
(u))
+ R^6_{m,a}(u),
\]
where 
\[
\ms_a(u,\bar u) = M_a(u,\bar u) + B^4_{m,a}(u).
\]
To prove \eqref{unif-0} we integrate the 
above density-flux relation in $t,x$ to obtain:
\begin{equation}\label{en-ident}
\left. \int M_a(u) + B^4_{m,a}(u) \, dx  \right|_0^T = \int_{0}^T \int_\R R^6_{m,a}(u) \ dx dt.
\end{equation}
Finally, we can estimate the contributions of $B^4_{m,a}$ and $R^6_{m,a}$
using Lemma~\ref{l:B4-multi}, respectively Lemma~\ref{l:R6-AB}.

\begin{remark}
For later use, we observe that once the energy bounds \eqref{uk-ee}
have been established, then they can be used instead of 
the bootstrap assumption \eqref{uk-ee-boot} in the proof
of Lemma~\ref{l:B4-multi}. This leads to a stronger form of 
\eqref{b4-ma}, \eqref{b4-pa}, with the constant $C$ removed:
\begin{equation}\label{B4-in-L1-re}
\| B^4_{m,A}(u)\|_{L^\infty_t L^1_x} + \| B^4_{p,A}(u)\|_{L^\infty_t L^1_{x}}
\lesssim  c_A^2 \epsilon^4.
\end{equation}
\end{remark}

\subsection{ The localized interaction Morawetz}
Our objective here is to prove the bounds \eqref{uk-se} 
and \eqref{uk-bi} using our bootstrap assumptions. 
By the Galilean invariance it suffices to do this 
at $k=0$. This will be achieved using our interaction Morawetz  identity \eqref{interaction-xi} with $v=u$ and with 
$a$ localized at frequency $0$, exactly as in \eqref{choose-a}.
For such $a$ we can simply set $\xi_0=0$.
It will suffice to estimate the quantities in \eqref{interaction-xi} as follows:

\begin{equation}\label{Ia-bound}
|\bI_{a}(u,u)| \lesssim \epsilon^4 c_0^4, 
\end{equation}
\begin{equation}\label{J4-formula}
  \bJ^4_{a}(u,u) \approx  \| \partial_x |A_0(D) u|^2\|_{L^2_x}^2   , 
\end{equation}
\begin{equation}\label{J6-bound}
\int_0^T \bJ^6_a(u,u)\, dt  \approx  \|  A_0(D)^\frac23 u\|_{L^6_{t,x}}^6
+ O(\epsilon^5 C^6 c_0^4) ,
\end{equation}
\begin{equation}
\int_0^T \bJ^8_a(u,u)\, dt = O(\epsilon^5  C^6 c_0^4), 
\end{equation}
\begin{equation}
\int_0^T\bK^8_a(u,u)\, dt = O(\epsilon^5 C^8  c_0^4). 
\end{equation}
This allows  us to estimate the localized interaction Morawetz term, as well as the localized $L^6$ norm
as in \eqref{uk-se} and \eqref{uk-bi}, provided that $\epsilon$ is small enough.
 There is nothing to do for $\bJ^4_{a}$ so we consider the remaining contributions:

\subsubsection{The $\bI_a$ bound}
The interaction Morawetz functional
$\bI_a$ is as in \eqref{Ia-sharp-def}, with $\xi_0=0$,
\begin{equation}\label{Ia-sharp-def-re}
\bI_{a} =   \iint_{x > y} \ms_a(u)(x) \ps_{a}(v) (y) -  
\ps_{a}(u)(x) \ms_{a}(v) (y)\, dx dy
\end{equation}
with
\[
\ms_a(u)= M_a(u) + B^4_{m,a}(u),
\qquad 
\ps_a(u)= P_a(u) + B^4_{p,a}(u).
\]
For $B^4_{m,a}$ and $B^4_{p,a}$ we have the $L^{\infty}_tL^1_x$
bound \eqref{B4-in-L1-re}. For $M_a(u)$ and $P_a(u)$
we have the straightforward uniform in time bounds
\begin{equation}\label{M-L1}
\|M_a(u)\|_{L^\infty_t L^1_x} + \|P_a(u)\|_{L^\infty_t L^1_x} \lesssim \epsilon^2 c_a^2.
\end{equation}
Combining this with \eqref{B4-in-L1-re}, the estimate 
\eqref{Ia-bound} immediately follows.

\subsubsection{The $\bJ^6_a$ bound} This is a $6$-linear
expression whose expression we recall from \eqref{J6-def},
\begin{equation}\label{J6-def-re}
\bJ^6_{a} =  2 \int -( P_{a} B^4_{p,a}
+ P_{a} R^4_{m,a}) + ( 
M_a R^4_{p,a} +E_{a} B^4_{m,a})\, dx ,
\end{equation}
where again we have set $\xi_0 = 0$.

We first discuss the symbol localization properties 
for $\bJ^6_a$ with respect to the six entries at frequencies
$k_1$, $k_2$, $k_3$, $k_4$, $k_5$ and $k_6$.
Here we a-priori have two frequencies
close to $0$, say $k_5 = k_6=0$, namely those arising from $M_a$, $P_a$ and $E_a$, all of which have smooth and bounded symbols. In the symbols for $B^4_a$ and $R^4_a$, on the other hand, we have at least one frequency equal to zero, say $k_1=0$, and the near-diagonal property $\Delta^4 k= 0$.

Next we consider the size of the symbols, where we use 
Proposition~\ref{p:symbols}. This gives the following symbol
bounds regardless of the $p$ or $m$ index:
\[
|b^4_a| \lesssim \frac{1}{\la \delta k^{hi}\ra \la \delta k^{med}\ra}, \qquad 
|r^4_a| \lesssim \frac1{\la \delta k^{med}\ra},
\]
and similarly for their derivatives.
We split the analysis in two cases, depending on whether
all frequencies are equal (i.e. $\delta k^{hi} \lesssim 1$)
or not.
\bigskip

\emph{A. The case of separated
frequencies, $\delta k^{hi} \gg 1$.} To fix the notations,
suppose that $|k_2| \approx \delta k^{med} \approx n_1$
and $|k_3| \approx |k_4|\approx \delta k^{hi} \approx n_2$
where $n_1 \leq n_2$ represent dyadic scales.
Then we can apply two bilinear $L^2$ bounds \eqref{uk-bi-boot} for the frequency pairs $(k_1=0,k_4)$ and $(k_2,k_3)$ and simply estimate the $k_5$ and $k_6$
factors in $L^\infty$ by Bernstein's inequality. This yields
the bound for the corresponding portion of $J^6_a$
\[
\left|\int _0^T\bJ^{6,unbal}_a(u)\, dt \right| \lesssim
\epsilon^6 C^6 c_0^3 \sum_{k_2,k_3,k_4} \frac{1}{n_1 n_2} c_{k_2} c_{k_3} c_{k_4}.
\]
Since $k_4-k_3=k_2$, for fixed $k_2$ we can apply Cauchy-Schwartz inequality with respect to the $k_3$ and $k_4$ indices
to obtain
\[
\left|\int_0^T\bJ^{6,unbal}_a(u)\, dt \right| \lesssim
\epsilon^6 C^6 c_0^3 \sum_{|k_2| \approx n_1 \leq n_2} \frac{1}{n_1 n_2} c_{k_2} c_{n_2}^2.
\]
Finally, using the maximal function property for $c_{k_2}$
we arrive at
\[
\left|\int_0^T\bJ^{6,unbal}_a(u)\, dt\right| \lesssim
\epsilon^6 C^6 c_0^4 \sum_{n_2} \frac{\log n_2}{ n_2}  c_{n_2}^2,
\]
which suffices.
\bigskip

\emph{B. The case of equal
frequencies, $\delta k^{hi} \lesssim 1$.} Here 
we have $|k_j| \lesssim 1$ for all $j$, and the symbol
of $j^6_a$ is smooth and bounded.
The important feature here is the symbol of the 6-linear form $J^6_0$ on the diagonal
\[
\{ \xi_1 = \xi_2 = \xi_3 = \xi_4 = \xi_5 = \xi_6 \},
\]
which we would like to be positive. But we know this by \eqref{good-J6}, which shows that this equals
\[
j^6_a(\xi) = a_0^4(\xi) c(\xi, \xi, \xi).
\]

It follows that  we can write the symbol $j^6_a$ in the form
\[
j^6_a(\xi_1,\xi_2,\xi_3,\xi_4,\xi_5,\xi_6) = b_0(\xi_1) b(\xi_2) b(\xi_3) b(\xi_4) b(\xi_5) b(\xi_6)  + 
j^{6,rem}_a(\xi_1,\xi_2,\xi_3,\xi_4,\xi_5,\xi_6) ,
\]
where $b_0(\xi) = a_0(\xi)^\frac23 c(\xi, \xi, \xi)^\frac16$ and $j^{6,rem}_0$
vanishes when all $\xi$'s are equal. Then we can write
$j^{6,rem}_0$ as a linear combination of terms $\xi_{even} -\xi_{odd}$ with smooth coefficients. The first term 
yields the desired $L^6$ norm,
\[
\bJ^6_a(u) = \| B_0(D) u\|_{L^6_x}^6 + \bJ^{6,rem}_a.
\]
On the other hand the contribution   $\bJ^{6,rem}_a$ of the second term  be estimated using a bilinear $L^2$ bound \eqref{uk-bi-boot}, 
three $L^6$ bounds \eqref{uk-se-boot} and one $L^\infty$ via Bernstein's inequality,
\[
\left|\int_0^T\bJ^{6,rem}_a(u)\, dt \right| \lesssim \|J^{6,rem}_a(u)\|_{L^1_{t,x}}
\lesssim 
(C \epsilon^2 c_0^2) C^3(\epsilon c_0)^2 C \epsilon c_0 = C^5 \epsilon^5
c_0^5,
\]
which suffices.

\subsubsection{The bound for $\bJ^8_0$}
We recall that $\bJ^8_0$ has an expression of the form
\begin{equation}\label{J8-def-re}
\bJ^8_{a} =   \int
B^4_{m,a}(u) R^4_{p,a}(u) - R^4_{m,a}(u) B^4_{p,a}(u)
+ B^4_{m,a}(u) R^4_{p,a}(u) - R^4_{m,a}(u) B^4_{p,a}(u)
 \, dx, 
\end{equation}
see \eqref{J8-def} where we set $\xi_0 = 0$.
For this we need to show that 
\[
\left| \int_0^T\bJ^8_a\, dt\right| \lesssim \epsilon^6 c_0^4.
\]
This is an $8$-linear term which
has two factors, both of which 
are $4$-linear terms with output 
at frequency $0$ and one factor 
at frequency $0$. But the symbols are 
not the same, i.e. we have more decay in $B^4_a$ than  in $R^4_a$.

As usual, we localize the entries of $\bJ^8_a$ on the unit frequency scale and estimate each term separately.
We denote the four frequencies in $B^4_a$ by $k_1,k_2,k_3.k_4$
with $k_1 = 0$, and the four frequencies in $R^4_a$ by $l_1,l_2,l_3,l_4$ with $l_1=0$. These are constrained by the relations
$\Delta^4 k = 0$, $\Delta^4 l=0$. In addition, their symbols are bounded, along with their derivatives, as follows:
\[
|b^4_{a}| \lesssim \frac{1}{\la \delta k^{hi}\ra \la\delta k^{med}\ra},
\qquad |r^4_a| \lesssim \frac{1}{\la \delta l^{med}\ra}.
\]
 We consider several cases:

\bigskip

A) All eight frequencies are close to zero. Then we use six $L^6_{t,x}$ Strichartz bounds as in \eqref{uk-se-boot}, and two  $L^\infty$ bounds obtained from the energy via Bernstein's 
inequality.

\bigskip 

B) Some frequencies are away from zero. Denote by $n_1 \leq n_2$ the dyadic separations for the $k_j$ frequencies in  $B_4$, and by $o_1 \leq o_2$ the dyadic separations for the $l_j$ frequencies in $R_4$. We consider two cases depending on how $n_2$ and $o_2$ compare.

\bigskip

B1) $n_2 \lesssim  o_2$. Then the $R^4$ frequencies are in two $o_2$ separated clusters with distance below $o_1$ within each cluster. We use two bilinear $L^2$ bounds there, and $L^\infty$ bounds for all the $B^4_a$  factors 
to estimate
\[
\left|\int_0^T \bJ^8_a(u)\, dt\right| \lesssim \epsilon^8 C^6 c_0^2 \sum \frac{1}{n_1 n_2}
c_{k_2} c_{k_3} c_{k_4} \frac{1}{o_1 o_2} c_{l_2} c_{l_3} c_{l_4} .
\]
Suppose $k_2$ and $l_2$ are the smaller frequencies in each group, so that $|k_2| \approx n_1$ and $|l_2| \approx o_1$.  For fixed $k_2$ respectively $l_2$ we apply the Cauchy-Schwarz inequality for the pairs $(k_3,k_4)$, respectively $(l_3,l_4)$. We obtain 
\[
\left|\int_0^T \bJ^8_a(u)\, dt\right| \lesssim \epsilon^8 C^6 c_0^2 \sum \frac{1}{n_1 n_2}
c_{k_2} c_{n_2}^2 \frac{1}{o_1 o_2} c_{l_2} c_{o_2}^2  .
\]
Now we use the maximal function to also fix $k_2$ and $l_2$,
\[
\left|\int_0^T \bJ^8_a(u)\, dt\right| \lesssim \epsilon^8 C^6 c_0^4 \sum_{n_2 \leq o_2} \frac{\log n_2}{ n_2}
 c_{n_2}^2  \frac{\log o_2}{ o_2}
 c_{o_2}^2 \lesssim \epsilon^8 C^6 c_0^4.
\]

\bigskip

B2) $o_2 \ll  n_2$. Here we proceed exactly as before
but  using instead two bilinear $L^2$ bounds in $B^4_a$.
Following the same steps, we arrive at
\[
\left|\int_0^T\bJ^8_a(u)\, dt \right| \lesssim \epsilon^8 C^6 c_0^4 \sum_{o_2 \leq n_2} \frac{\log n_2}{ n_2^2}
 c_{n_2}^2 \log o_2
 c_{o_2}^2 \lesssim \epsilon^8 C^6 c_0^4.
\]
Here the denominators are unbalanced compared to the previous case, but in a favourable way.

\subsubsection{The bound for $\bK^8_a$} 
We recall that $\bK^8_a$ has the form
\begin{equation}\label{K8-def-re}
\begin{aligned}
\bK^8_a(u) = \iint_{x > y}  &\ \ms_a(u)(x)  R^6_{p,a}(u)(y)  + \ps_{a}(u)(y) R^6_{m,a}(u)(x) 
\\
& \ -
\ms_a(u)(y)  R^6_{p,a}(u)(x)  - \ps_{a}(u)(x) R^6_{m,a}(u)(y) \, dx dy.
\end{aligned}
\end{equation}
The time integral of $\bK^8_a(u)$ is estimated directly using the $L^1_{t,x}$ bound for $R^6$ in Lemma~\ref{l:R6-AB} and the uniform $L^1_x$ bound
for $\ms$ and $\ps$, provided by  Lemma~\ref{l:B4-multi} 
together with the simpler bound \eqref{M-L1}.

\subsection{Near parallel interactions}

Here we briefly discuss the bilinear $L^2$ bound \eqref{uab-bi}
in the case when the sets $A$ and $B$ are of size $\lesssim 1$
and at distance $\lesssim 1$. This can be viewed on one hand as a slight generalization of the argument in the previous subsection, where instead of $v = u$ we take $v = u(\cdot +x_0)$. The only difference in the proof is that, because of the translations, we can no longer use the defocusing property to control the sign of the diagonal $\bJ^{6}$ contribution. 
However, this is not a problem because the localized $L^6$ norm of $u_k$ has already been estimated in the previous subsection.

\subsection{The transversal bilinear \texorpdfstring{$L^2$}{}  estimate} 
Here we prove the bilinear $L^2$ bound \eqref{uab-bi}.
This repeats the same analysis as before, but
using the interaction Morawetz functional  associated to two separated frequency intervals $A$ and $B$, of size at most $n$ and  with $n$ separation. Here we no longer take $v = u$, and instead we let $v = u(\cdot+x_0)$. The parameter $x_0 \in \R$
is arbitrary and the estimates are uniform in $x_0$.

Since $x_0$ does not play any role in the analysis, we simply drop it from our notations. To further simplify the notations
in what follows, we take advantage of the Galilean invariance to translate the problem in frequency so that $0$ is roughly  half-way between the intervals  $A$ and $B$.
This will allow us to set $\xi_0=0$ in \eqref{interaction-bi},
and to assume that both $A$ and $B$ are within distance $n$
from the origin.
We consider mass $m_a$, $m_b$ and momentum forms $p_a$, $p_b$, where $a$ and $b$
are  bump functions, smooth on the
unit scale, selecting the sets $A$ and $B$. 

The  interaction functional  
  takes the form (see \eqref{interaction-bi})
\begin{equation}\label{interaction-bi-re}
 \bI_{AB}(u,v) = \iint_{x > y} \ms_a(u)(x) \ps_{b}(v)(y)    - \ps_{a}(u)(x) \ms_b(v)(y) \,dx dy .
\end{equation}
Its time derivative is given, see \eqref{interaction-xi}, by
\begin{equation}\label{interaction-xi-re}
\frac{d}{dt} \bI_{AB} =  \bJ^4_{AB} + \bJ^6_{AB} + \bJ^8_{AB} + \bK^8_{AB} .
\end{equation}
Following the same pattern as in the earlier case of the localized interaction Morawetz case, we will estimate each of these
terms as follows:
\begin{equation}\label{IAB-bound}
|\bI_{AB}(u,v)| \lesssim  n \epsilon^4 c_A^2 c_B^2  ,
\end{equation}
\begin{equation}\label{J4AB-formula}
  \bJ^4_{AB}(u,v) \approx  \| \partial_x (u_A \bv_B)\|_{L^2_x}^2   , 
\end{equation}
\begin{equation}\label{J6AB-bound}
\left|\int_0^T\bJ^6_{AB}\, dt \right| \lesssim n (\epsilon^6 C^6 + \epsilon^4) c_A^2  c_B^2 ,
\end{equation}
\begin{equation}\label{J8AB-bound}
\left|\int _0^T\bJ^8_{AB}\, dt\right| \lesssim n \epsilon^8  C^8  c_A^2 c_B^2 ,
\end{equation}
\begin{equation}\label{K8AB-bound}
\left| \int_0^T \bK^8_{AB}\, dt \right| \lesssim n \epsilon^6  C^8  c_A^2 c_B^2 .
\end{equation}

\subsubsection{ The fixed time estimate for $\bI_{AB}$}
Here we prove the bound \eqref{IAB-bound}, which is a consequence
of fixed time $L^1$ estimates for the energy densities, namely
\begin{equation}\label{MPA-bd}
\| \ms_a(u)\|_{L^1_x} \lesssim \epsilon^2 c_A^2, 
\qquad
\| \ps_a(u)\|_{L^1_x} \lesssim n \epsilon^2 c_A^2,
\end{equation}
and the similar estimates with $a$ replaced by $b$ and $u$ replaced by $v$. 
This is obvious for the quadratic part of the above densities,
where we note that the $n$ factor for the momentum bound arises due to the distance $o(n)$ between the set $A$ and the origin. 
It remains to consider the  quartic terms, where we can use Lemma~\ref{l:B4-multi} together with Corollary~\ref{c:B4-multi}.

\subsubsection{ The bound for $\bJ^6_{AB}$}
Here we prove the bound for $\bJ^6_{AB}$ in \eqref{J6AB-bound}. 
We recall that $\bJ^6_{AB}$ has the form
\[
\bJ^6_{AB} = \int M_a(u) R^4_{p,b}(v) - P_b(v) R^4_{m,a}(u)
+ B^4_{m,a}(u) E_b(v) - B^4_{p,b}(v) P_a(u) - \text{symmetric}
\, dx,
\]
where the symmetric term is obtained by interchanging the indices $a$
and $b$, and also $u$ and $v$. The symbols for the $M$, $P$ and $E$ factors have size 
$1$, $n$ and $n^2$ respectively, with a similar balance between the $B^4_m$
and $B^4_p$ terms, respectively   the $R^4_m$
and $R^4_p$ terms. So it suffices to consider one $R^4$ term
and one $B^4$ term.

\bigskip

A) The $B^4$ term $B^4_{m,a}(u) E_b(v)$. Here we denote by $l_1,l_2$ the $E_b$ frequencies 
and by $k_1,k_2,k_3, k_4$ the $B^4_{m,a}$ frequencies 
where 
\[
 \Delta^2 l + \Delta^4 k  = 0 .
\]
The symbol for $E_b$ has size $n^2$, with both frequencies in $B$. The symbol 
for $B^4_{m,a}(u)$ has size $( \la \delta k^{med}\ra \la \delta k^{hi}\ra)^{-1}$ and support 
in the region where $|\Delta^4 k| \ll 1 + k^{med}$, and at least one of the frequencies is in $A$. We denote the dyadic sizes of $k^{med}$ and $k^{hi}$ by $n_1 \leq  n_2$. Without any loss in generality we may assume that
$k_1,k_2,k_3,k_4$ are chosen so that 
\begin{equation}\label{which-k-R}
k_1 \in A, \qquad 
|k_1-k_2| \approx n_1, \quad |k_1-k_3| \approx n_2, \qquad 
|k_1-k_4| \approx n_2, \quad
|k_3-k_4| \approx n_1.
\end{equation}

Depending on the size of $n$ relative to $n_1,n_2$ we consider two cases:

\bigskip

A1) $ n_2 \ll n$. Since $A$ and $B$ are $n$-separated, within the set of six frequencies we can find two pairs of $n$ -separated frequencies. Then we can apply twice the bilinear $L^2$ bound and estimate the remaining factors in $L^\infty$. We arrive at the frequency envelope bound
\[
\left|\int_0^T\bJ^6_{AB}\, dt\right| \lesssim  \epsilon^6 C^6 n  \sum \frac{1}{n_1 n_2}
c_{l_1} c_{l_2} c_{k_1} c_{k_2} c_{k_3} c_{k_4},
\]
where the summation indices are restricted as discussed above. 
Then, applying the Cauchy-Schwarz inequality 
for the pair $(l_1,l_2)$  we obtain
\[
\begin{aligned}
\left|\int_0^T\bJ^6_{AB}\, dt \right| \lesssim & \  \epsilon^6 C^6 n  \sum \frac{1}{n_1 n_2} 
\sum_{|\Delta^4 k| < n_1} c_{k_1} c_{k_2} c_{k_3} c_{k_4} 
 \sum_{l_1,l_2 \in B}^{\Delta l =- \Delta^4 k} c_{l_1} c_{l_2}
\\
\lesssim & \  \epsilon^6 n C^6 c_B^2 \sum \frac{1}{ n_1 n_2} 
\sum_{|\Delta^4 k| < n_1} c_{k_1} c_{k_2} c_{k_3} c_{k_4}.
\end{aligned}.
\]
Hence it remains to estimate the last sum above as follows:
\begin{equation}\label{S-bd}
S_A: =   \sum_{n_1 < n_2} \frac{1}{ n_1 n_2} 
\sum_{D} c_{k_1} c_{k_2} c_{k_3} c_{k_4} \lesssim c_A^2,  
\end{equation}
where the summation set $D$ is described by \eqref{which-k}.
In this estimate the parameter $n$ no longer appears.
Recalling that $k_1 \in A$, we fix $k_1$ and split 
\[
S_A  = \sum_{k_1 \in A} c_{k_1} S_{k_1}, \qquad 
S_{k_1}:=\sum_{D}\sum_{n_1 < n_2} \frac{1}{ n_1 n_2} 
\sum_{D}  c_{k_2} c_{k_3} c_{k_4}.
\]
Then it suffices to show that 
\[
S_{k_1} \lesssim c_{k_1},
\]
which is exactly the bound \eqref{fix-k1} proved earlier.

\medskip

A2) $ n_2 \gtrsim n$. This time, within the set of four $k$ frequencies we can find two pairs of $n_2$ -separated frequencies. Applying twice the bilinear $L^2$ bound and estimating the remaining factors in $L^\infty$ we arrive at 
\[
\left|\int_0^T\bJ^6_{AB}\, dt\right| \lesssim  \epsilon^6 C^6 n^2  \sum \frac{1}{n_1 n_2^2}
c_{l_1} c_{l_2} c_{k_1} c_{k_2} c_{k_3} c_{k_4},
\]
Applying the Cauchy-Schwarz inequality 
for the pair $(l_1,l_2)$  now yields
\[
\begin{aligned}
\left|\int_0^T\bJ^6_{AB}\, dt \right| \lesssim & \  \epsilon^6 C^6 n^2 c_B^2 \sum \frac{1}{n_1 n_2^2} 
\sum_{|\Delta^4 k| < n} c_{k_1} c_{k_2} c_{k_3} c_{k_4} 
\end{aligned}.
\]
Since $n \lesssim n_2$, we can conclude again using the bound \eqref{S-bd} which was already proved in (A1).

\bigskip

B. The $R^4$ terms are also all similar, so to fix the notations we will discuss the expression $P_b(u) R^4_{m,a}(u)$. We denote again the six frequencies by $l_1,l_2$ for $P_b$, respectively by $k_1,k_2,k_3, k_4$  for $B^4_{m,a}$. The symbol of $p_B$ is supported in $B \times B$ and has size 
$n$. The symbol $R^4_{m,a}$ has size 
\[
|r^4_{m,a}([k])| \lesssim \frac{n + \delta k^{hi}}{ \la \delta k^{med}\ra \la\delta k^{hi}\ra}.
\]
The bound for the portion containing the $n$ term in the denominator is 
identical to the one in case A, so in the sequel we dismiss this term 
and simplify the above bound to 
\[
|r^4_{m,a}([k])| \lesssim \frac{1}{ \delta k^{med}}.
\]
Retaining the notations $n_1 \leq n_2 $ for the dyadic sizes of $\delta k^{med}$ and $\delta k^{hi}$, we may also restrict our analysis to the case when $n_2 \gg n$. This is similar to case A above. We get the  better $n_2^{-1}$
factor from the bilinear $L^2$ bounds, which allows us to reduce the problem to proving exactly the  bound \eqref{S-bd}, but for a larger set of indices
\begin{equation}\label{which-k-extra}
k_1 \in A, \qquad 
|k_1-k_2| \lesssim n_1, \quad |k_1-k_3| \approx n_2, \qquad 
|k_1-k_4| \approx n_2, \quad
|k_3-k_4| \lesssim n_1.
\end{equation}
But this still follows from \eqref{fix-k1}.

\subsubsection{ The bound for $\bJ^8_{AB}$}
Here we prove the bound \eqref{J8AB-bound}.
We recall that $\bJ^8_{AB}$ has the form
\[
\bJ^8_{AB} = \iint B^4_{m,a}(u) R^4_{p,b}(v) - B^4_{p,b} (v) R^4_{m,a}(u)
+ B^4_{m,b}(v) R^4_{p,a}(u) - B^4_{p,a}(u) R^4_{m,b}(v) 
\, dxdt.
\]
All terms here are similar, so it suffices to consider the first one. To avoid a lengthy proof  which would largely repeat the arguments in the proof of \eqref{J6AB-bound},
we make a simple observation, namely that the proof of 
the bound for this term becomes a corollary of the previous bound if we can establish a representation
\[
B^4_{m,a}(u) \approx \sum_{l_1,l_2 \in A}  u_{l_1} w_{l_2},
\]
so that, for each $k$ which is $M$-separated from $l_2$,  the function $w_{l_2}$ satisfies a bilinear $L^2$ bound of the form 
\begin{equation}\label{vl-uk}
\| w_{l_2} u_k \|_{L^2_{t,x}} \lesssim M^{-\frac12} C^4 \epsilon^4 c_{l_2} c_{k} .
\end{equation}
If that is true, then $w_{l_2}$ would play exactly the role 
of $u_{l_2}$ in the $\bJ^6_{AB}$ estimate.

Indeed, we may represent 
\[
B^4_{m,a}(u) = \sum_{l_1,l_2 \in A}  \sum_{k_2 - k_3+k_4 = l_2}
 B^4_{m,a}(u_{l_1},u_{k_2},u_{k_3}, u_{k_4}) .
\]
Here the symbol for $B^4_{m,a}$ and its derivatives have size $\lesssim \dfrac{1}{n_1 n_2}$ in a unit region around frequency $(l_1,k_2,k_3,k_4)$. Hence, we may separate variables
and represent $B^4_{m,a}(u_{l_1},u_{k_2},u_{k_3}, u_{k_4})$
as the sum of a rapidly convergent series
\[
B^4_{m,a}(u_{l_1},u_{k_2},u_{k_3}, u_{k_4}) = 
\sum_{j} D^{j} u_{l_1} B^{4,j}_{m,a}(u_{k_2},u_{k_3}, u_{k_4}) : = \sum D^j u_{l_1} w^j_{l_2},
\]
where the symbols for $D^j$, respectively $B^{4,j}_{m,a} $
have unit size, respectively $\lesssim \frac{1}{n_1 n_2}$
with rapid decay in $j$. Then it remains to prove the estimate \eqref{vl-uk} for the functions $w^j_{l_2}$. 

Indeed, at least one of the $k$'s must be $M$-separated from $k$, so using a bilinear $L^2$ bound  we have 
\[
\| w^j_{l_2} u_k \|_{L^2} \lesssim M^{-\frac12} \epsilon^4 C^4  c_{k} j^{-10} 
\sum_{k_2 - k_3+k_4 = l_2} \frac{1}{n_1 n_2} c_{k_2} c_{k_3} c_{k_4}.
\]
It remains to estimate the last sum. Suppose $k_2$ is within distance $n_1$  from $l_2$, then we use the maximal function to estimate
\[
\sum_{k_2 - k_3+k_4 = l_2} \frac{1}{n_1 n_2} c_{k_2} c_{k_3} c_{k_4}
\lesssim c_{l_2} \sup_{k_2} \sum_{k_2 - k_3+k_4 = l_2} \frac{1}{n_2}  c_{k_3} c_{k_4}
\lesssim c_{l_2},
\]
as needed.

\subsubsection{ The bound for   $\bK^8_{AB}$} 
This is immediate by combining the bound 
\eqref{MPA-bd} with the $R^6$ bounds in 
Lemma~\ref{l:R6-AB} and Corollary~\ref{c:R6-AB}.

\section{Global bilinear and Strichartz estimates}

Our objective in this last section is to supplement the unit frequency scale bilinear $L^2$ and Strichartz estimates
with their more global counterparts:

\begin{theorem}
The global small data solutions $u$ for \eqref{nls}
in Theorem~\ref{t:boot} satisfy the following bounds:

\begin{itemize}
    \item Strichartz estimate:
    \begin{equation}
    \| u\|_{L^6_{t,x}}^6 \lesssim \epsilon^4,
    \end{equation}
    \item Bilinear $L^2$ bound:
    \begin{equation}\label{main-bi-re}
    \|\partial_x |u|^2 \|_{L^2_t H^{-\frac12}_x}^2 \lesssim \epsilon^4.
    \end{equation}
\end{itemize}
\end{theorem}

\begin{proof}
We successively consider the two estimates:
\bigskip

\emph {A. The global \texorpdfstring{$L^6$}{} bound.}
We prove the global $L^6$ bound using  the previous localized estimates. We aim to estimate the integral 
\[
I = \iint_{\R \times \R} |u|^6 \, dx dt
\]
by taking a suitable frequency decomposition. Given six unit frequency regions indexed by $k_1$, $k_2$, $k_3$, $k_4$, $k_5$
and $k_6$, they can only contribute to the above integral 
iff $\Delta^6 k = 0$. We divide them as follows:
\begin{enumerate}
    \item The diagonal case $|k_i - k_j| \lesssim 1$.

\item The nondiagonal case. we index these frequencies 
by the dyadic size $n \gg 1$ of the set of frequencies, 
i.e. so that 
\[
\max |k_i-k_j| \approx n.
\]
Within this range, we organize frequencies in intervals $A_1, \cdots A_6$ of size $n/100$. Of these intervals, at least two pairs must be $n$-separated in order to contribute to the above integral.
\end{enumerate}
Based on this, we split $I$ as 
\[
I = I_0 + \sum_{n} I_n,
\]
where
\[
I_0 = \sum_{|k_i - k_j| \lesssim 1 } \iint u_{k_1} \bu_{k_2}
u_{k_3} \bu_{k_4} u_{k_5} \bu_{k_6} \, dx dt,
\]
and 
\[
I_n = \sum \iint u_{A_1} \bu_{A_2}
u_{A_3} \bu_{A_4} u_{A_5} \bu_{A_6} \, dxdt,
\]
where the last sum is indexed over the sets $A_j$ of size $n/100$, with largest 
distance $\approx n$ and at least two distances $\geq n/10$.

For the diagonal part we use the $L^6$ bound \eqref{uk-se}
to estimate
\[
|I_0| \lesssim \epsilon^4 \sum_k c_k^4 \lesssim \epsilon^4,
\]
which suffices.

For the off-diagonal part we apply two bilinear $L^2$
bounds for the separated intervals (gaining $n^{-\frac12}$ each time) and two  $L^\infty$ bounds via Bernstein's inequality 
(losing $n^\frac12$ each time) to bound the corresponding term by
\[
|I_n| \lesssim \epsilon^6 \sum c_{A_1} c_{A_2} c_{A_3} c_{A_4} c_{A_5} c_{A_6}.
\]
We retain only the separated parts and apply Cauchy-Schwarz inequality
to estimate 
\[
|I_n| \lesssim \epsilon^6 \sum_{d(A_1,A_2) > n/10} 
c_{A_1}^2 c_{A_2}^2 \lesssim \epsilon^6 \sum_{|k_1-k_2| \approx n} c_{k_1}^2 c_{k_2}^2.
\]
Then summation over $n$ yields 
\[
\sum_n |I_n| \lesssim \epsilon^6 \sum_{k_1,k_2}  c_{k_1}^2 c_{k_2}^2 \lesssim \epsilon^6,
\]
which again suffices.

\bigskip

\emph{B. The global bilinear \texorpdfstring{$L^2$}{} bound.}
Here we prove the estimate \eqref{main-bi-re}. 
Expanding relative to the dyadic difference $n$ of the two input frequencies we have
\[
\partial_x (|u|^2) = \partial_x w_0 + \sum_n \partial_x w_n,
\]
where
\[
w_0 = \sum_{|k_1-k_2| \lesssim 1}
u_{k_1} \bu_{k_2},
\]
\[
w_n = \sum_{|A_1|,|A_2| \approx n}^{d(A_1,A_2) \approx n} u_{A_1} \bu_{A_2}.
\]
We use \eqref{uk-bi} to estimate $w_0$ as 
\[
\|\partial_x w_0\|_{L^2_{t,x}}^2 \lesssim \epsilon^4 \sum_k c_k^4 \lesssim \epsilon^4. 
\]

On the other hand for $w_n$ we get 
\[
\|\partial_x w_n \|_{L^2_t H^{-\frac12}_x}^2 \lesssim  n
\sum_{|A_1|,|A_2| \approx n}^{d(A_1,A_2) \approx n}
\sum_{|A_3|,|A_4| \approx n}^{d(A_3,A_4) \approx n}
\int u_{A_1} \bu_{A_2}u_{A_3} \bu_{A_4} \, dx.
\]
Denoting by $n_0 \geq n$ the largest distance between two $A_j$'s we have two pairs of intervals with separation $O(n_0)$
therefore, applying twice the bilinear $L^2$ bound we obtain
\[
\|\partial_x w_n \|_{H^{-\frac12}}^2 \lesssim \epsilon^4  \sum_{n_0 \geq  n} \frac{n}{n_0}
\sum_{|A_1|,|A_2| \approx n}^{d(A_1,A_2) \approx n}
\sum_{|A_3|,|A_4| \approx n}^{d(A_3,A_4) \approx n}
 c_{A_1} c_{A_2} c_{A_3} c_{A_4} .
\]
We separate the cases when $n_0 \approx n$ and $n_0 \gg n$.
In the first, diagonal case we simply bound the corresponding part of the sum by 
\[
\epsilon^4 \sum_{|A_1|,|A_2| \approx n}^{d(A_1,A_2) \approx n}
c_{A_1}^2 c_{A_2}^2.
\]
In the off-diagonal case we apply  Cauchy-Schwarz inequality separately for the pairs $A_1,A_2$ and $A_3,A_4$ to obtain a bound
\[
\epsilon^4 \frac{n}{n_0} \sum_{|B_1|,|B_2| \approx n_0}^{d(B_1,B_2) \approx n_0}
c_{B_1}^2 c_{B_2}^2.
\]
Incorporating the first case into the second  we arrive at 
\[
\|\partial_x w_n \|_{L^2_t H^{-\frac12}_x}^2 \lesssim \epsilon^4  \sum_{n_0 \geq  n} \frac{n}{n_0}\sum_{|B_1|,|B_2| \approx n_0}^{d(B_1,B_2) \approx n_0}
c_{B_1}^2 c_{B_2}^2.
\]

Finally, using orthogonality in frequency we have 
\[
\begin{aligned}
\|\sum_n \partial_x w_n \|_{L^2_t H^{-\frac12}_x}^2 \lesssim & \   \epsilon^4 \sum_n \sum_{n_0 \geq  n} \frac{n}{n_0}\sum_{|B_1|,|B_2| \approx n_0}^{d(B_1,B_2) \approx n_0}
c_{B_1}^2 c_{B_2}^2
\\
 \lesssim & \  \epsilon^4 \sum_{n_0} \sum_{|B_1|,|B_2| \approx n_0}^{d(B_1,B_2) \approx n_0} c_{B_1}^2 c_{B_2}^2
\\
 \lesssim & \  \epsilon^4 .
\end{aligned}
\]
The proof of the theorem is concluded.

\end{proof}


\begin{thebibliography}{10}
	
	\bibitem{AD}
	Thomas Alazard and Jean-Marc Delort.
	\newblock Global solutions and asymptotic behavior for two dimensional gravity
	water waves.
	\newblock {\em Ann. Sci. \'{E}c. Norm. Sup\'{e}r. (4)}, 48(5):1149--1238, 2015.
	
	\bibitem{IST-focusing}
	Michael Borghese, Robert Jenkins, and Kenneth D. T.-R. McLaughlin.
	\newblock Long time asymptotic behavior of the focusing nonlinear
	{S}chr\"{o}dinger equation.
	\newblock {\em Ann. Inst. H. Poincar\'{e} C Anal. Non Lin\'{e}aire},
	35(4):887--920, 2018.
	
	\bibitem{MR2527809}
	J.~Colliander, M.~Grillakis, and N.~Tzirakis.
	\newblock Tensor products and correlation estimates with applications to
	nonlinear {S}chr\"{o}dinger equations.
	\newblock {\em Comm. Pure Appl. Math.}, 62(7):920--968, 2009.
	
	\bibitem{I-method}
	J.~Colliander, M.~Keel, G.~Staffilani, H.~Takaoka, and T.~Tao.
	\newblock Almost conservation laws and global rough solutions to a nonlinear
	{S}chr\"{o}dinger equation.
	\newblock {\em Math. Res. Lett.}, 9(5-6):659--682, 2002.
	
	\bibitem{MR2053757}
	J.~Colliander, M.~Keel, G.~Staffilani, H.~Takaoka, and T.~Tao.
	\newblock Global existence and scattering for rough solutions of a nonlinear
	{S}chr\"{o}dinger equation on {$\mathbb R^3$}.
	\newblock {\em Comm. Pure Appl. Math.}, 57(8):987--1014, 2004.
	
	\bibitem{MR2415387}
	J.~Colliander, M.~Keel, G.~Staffilani, H.~Takaoka, and T.~Tao.
	\newblock Global well-posedness and scattering for the energy-critical
	nonlinear {S}chr\"{o}dinger equation in {$\mathbb R^3$}.
	\newblock {\em Ann. of Math. (2)}, 167(3):767--865, 2008.
	
	\bibitem{I-method2}
	J.~Colliander, M.~Keel, G.~Staffilani, H.~Takaoka, and T.~Tao.
	\newblock Resonant decompositions and the {$I$}-method for the cubic nonlinear
	{S}chr\"{o}dinger equation on {$\mathbb R^2$}.
	\newblock {\em Discrete Contin. Dyn. Syst.}, 21(3):665--686, 2008.
	
	\bibitem{IST}
	Percy Deift and Xin Zhou.
	\newblock Long-time asymptotics for solutions of the {NLS} equation with
	initial data in a weighted {S}obolev space.
	\newblock {\em Comm. Pure Appl. Math.}, 56(8):1029--1077, 2003.
	\newblock Dedicated to the memory of J\"{u}rgen K. Moser.
	
	\bibitem{D}
	Jean-Marc Delort.
	\newblock Semiclassical microlocal normal forms and global solutions of
	modified one-dimensional {KG} equations.
	\newblock {\em Ann. Inst. Fourier (Grenoble)}, 66(4):1451--1528, 2016.
	
	\bibitem{MR3483476}
	Benjamin Dodson.
	\newblock Global well-posedness and scattering for the defocusing, {$L^2$}
	critical, nonlinear {S}chr\"{o}dinger equation when {$d=1$}.
	\newblock {\em Amer. J. Math.}, 138(2):531--569, 2016.
	
	\bibitem{MR3625190}
	Benjamin Dodson.
	\newblock Global well-posedness and scattering for the defocusing,
	mass-critical generalized {K}d{V} equation.
	\newblock {\em Ann. PDE}, 3(1):Paper No. 5, 35, 2017.
	
	\bibitem{HN}
	Nakao Hayashi and Pavel~I. Naumkin.
	\newblock Asymptotics for large time of solutions to the nonlinear
	{S}chr\"{o}dinger and {H}artree equations.
	\newblock {\em Amer. J. Math.}, 120(2):369--389, 1998.
	
	\bibitem{HN1}
	Nakao Hayashi and Pavel~I. Naumkin.
	\newblock Large time asymptotics for the fractional nonlinear {S}chr\"{o}dinger
	equation.
	\newblock {\em Adv. Differential Equations}, 25(1-2):31--80, 2020.
	
	\bibitem{IT-NLS}
	Mihaela Ifrim and Daniel Tataru.
	\newblock Global bounds for the cubic nonlinear {S}chr\"{o}dinger equation
	({NLS}) in one space dimension.
	\newblock {\em Nonlinearity}, 28(8):2661--2675, 2015.
	
	\bibitem{IT-g}
	Mihaela Ifrim and Daniel Tataru.
	\newblock Two dimensional water waves in holomorphic coordinates {II}: {G}lobal
	solutions.
	\newblock {\em Bull. Soc. Math. France}, 144(2):369--394, 2016.
	
	\bibitem{IT-c}
	Mihaela Ifrim and Daniel Tataru.
	\newblock The lifespan of small data solutions in two dimensional capillary
	water waves.
	\newblock {\em Arch. Ration. Mech. Anal.}, 225(3):1279--1346, 2017.
	
	\bibitem{IT-BO}
	Mihaela Ifrim and Daniel Tataru.
	\newblock Well-posedness and dispersive decay of small data solutions for the
	{B}enjamin-{O}no equation.
	\newblock {\em Ann. Sci. \'{E}c. Norm. Sup\'{e}r. (4)}, 52(2):297--335, 2019.
	
	\bibitem{KP}
	Jun Kato and Fabio Pusateri.
	\newblock A new proof of long-range scattering for critical nonlinear
	{S}chr\"{o}dinger equations.
	\newblock {\em Differential Integral Equations}, 24(9-10):923--940, 2011.
	
	\bibitem{KT}
	Herbert Koch and Daniel Tataru.
	\newblock Energy and local energy bounds for the 1-d cubic {NLS} equation in
	{$H^{-\frac14}$}.
	\newblock {\em Ann. Inst. H. Poincar\'{e} Anal. Non Lin\'{e}aire},
	29(6):955--988, 2012.
	
	\bibitem{LLS}
	Hans Lindblad, Jonas L\"{u}hrmann, and Avy Soffer.
	\newblock Asymptotics for 1{D} {K}lein-{G}ordon equations with variable
	coefficient quadratic nonlinearities.
	\newblock {\em Arch. Ration. Mech. Anal.}, 241(3):1459--1527, 2021.
	
	\bibitem{LS}
	Hans Lindblad and Avy Soffer.
	\newblock Scattering and small data completeness for the critical nonlinear
	{S}chr\"{o}dinger equation.
	\newblock {\em Nonlinearity}, 19(2):345--353, 2006.
	
	\bibitem{PV}
	Fabrice Planchon and Luis Vega.
	\newblock Bilinear virial identities and applications.
	\newblock {\em Ann. Sci. \'{E}c. Norm. Sup\'{e}r. (4)}, 42(2):261--290, 2009.
	
	\bibitem{MR2288737}
	E.~Ryckman and M.~Visan.
	\newblock Global well-posedness and scattering for the defocusing
	energy-critical nonlinear {S}chr\"{o}dinger equation in {$\mathbb R^{1+4}$}.
	\newblock {\em Amer. J. Math.}, 129(1):1--60, 2007.
	
	\bibitem{Tao-WM}
	Terence Tao.
	\newblock Global regularity of wave maps. {II}. {S}mall energy in two
	dimensions.
	\newblock {\em Comm. Math. Phys.}, 224(2):443--544, 2001.
	
	\bibitem{Tao-BO}
	Terence Tao.
	\newblock Global well-posedness of the {B}enjamin-{O}no equation in {$H^1({\bf
			R})$}.
	\newblock {\em J. Hyperbolic Differ. Equ.}, 1(1):27--49, 2004.
	
\end{thebibliography}

\bibliographystyle{plain}

\end{document}